\documentclass[letterpaper,final,10pt]{amsart}
\usepackage{inputenc}

\usepackage{amsmath}
\usepackage{amsfonts}
\usepackage{amssymb}
\usepackage{graphicx}
\usepackage{mathrsfs}
\usepackage[obeyspaces]{url}
\usepackage[all,cmtip,curve,knot]{xy}
\usepackage{tikz}
\usetikzlibrary{arrows}
\tikzstyle{block}=[draw opacity=0.7,line width=1.4cm]
\newtheorem{theorem}{Theorem}

\newtheorem{corollary}{Corollary}

\newtheorem{definition}{Definition}

\newtheorem{lemma}{Lemma}

\newtheorem{remark}{Remark}

\numberwithin{equation}{section}
\numberwithin{theorem}{section}
\numberwithin{lemma}{section}
\numberwithin{corollary}{section}
\numberwithin{definition}{section}
\numberwithin{example}{section}
\numberwithin{remark}{section}
\numberwithin{property}{section}
\numberwithin{proposition}{section}

\newcommand{\NS}[3][1]{#2_#1,\ldots,#2_#3}

\newcommand{\lift}[3][\pi]{\xymatrix{#2 \ar@{~>}[r]_{#1} & #3}}

\newcommand{\diag}[2][]{\mathrm{diag}_{#1}\left[ #2 \right]}

\newcommand{\U}[1]{\mathbb{U}(#1)}

\newcommand{\I}{\mathbf{1}}

\newcommand{\RR}{\mathbb{R}}
\newcommand{\CC}{\mathbb{C}}
\newcommand{\ZZ}{\mathbb{Z}}

\newcommand{\TT}[1][1]{{\mathbb{T}^{#1}}}

\newcommand{\disk}[1][2]{{\mathbb{D}^{#1}}}

\begin{document}
\title{On Uniform Connectivity of Algebraic Matrix Sets}
\author{Fredy Vides}
\address{Department of Applied Mathematics, School of Mathematics, Faculty of Science, 
Universidad Nacional Aut\'onoma de Honduras, Ciudad Universitaria, Tegucigalpa, Honduras.}

\email{fredy.vides@unah.edu.hn}

\keywords{Joint spectrum, Matrix function, Functional calculus, Algebraic set, Matrix path, Eigenvalue clustering.}

\subjclass[2010]{47N40, 47A58 (primary) and 15A60, 15A27 (secondary).} 

\date{\today}

\begin{abstract}
In this document we study the uniform local path connectivity of sets of $m$-tuples of pairwise commuting normal matrices with some additional constraints. 

More specifically, given given $\varepsilon>0$, a fixed metric $\eth$ in ${M_n(\mathbb{C})}^m$ induced by the operator norm $\|\cdot\|$, any collection of $r$ non-constant multivariable polynomials $p_1(x_1,\ldots,x_m),\ldots,p_r(x_1,\ldots,x_m)$ over $\mathbb{C}$ with finite zero set $\mathbf{Z}(p_1,\ldots,p_r)\subset \mathbb{C}^m$, and any $m$-tuple $\mathbf{X}=(X_1,\ldots,X_m)$ in the set $\mathbb{ZD}_n^m(p_1,\ldots,p_r)\subseteq M_n^m(\mathbb{C})$, of pairwise commuting normal matrix contractions such that,  $\|p_j(Y_1,\ldots,Y_m)\|=0$ for each $(Y_1,\ldots,Y_m)\in \mathbb{ZD}_n^m(p_1,\ldots,p_r)$ and each $1\leq j\leq r$. We prove the existence of paths between arbitrary $m$-tuples, that lie in the intersection of $\mathbb{ZD}_n^m(p_1,\ldots,p_r)$, and the $\delta$-ball $B_\eth(\mathbf{X},\delta)$ centered at $\mathbf{X}$ for some $\delta>0$, with respect to $\eth$. 

Two of the key features of these matrix paths is that $\delta$ can be chosen independent of $n$, and that they are contained in the intersection of $B_\eth(\mathbf{X},\varepsilon)$ and $\mathbb{ZD}_n^m(p_1,\ldots,p_r)$.

Some connections with the approximation theory for matrix functions of several matrix variables, are studied as well.
\end{abstract}

\maketitle

\section{Introduction}
\label{intro}

In this document we study the uniform local path connectivity of sets of $m$-tuples of commuting normal matrices with some additional geometric and algebraic constraints in their joint spectra. 

Let $\varepsilon>0$ be given. Given a fixed metric $\eth$ in ${M_n(\mathbb{C})}^m$ induced by the operator norm $\|\cdot\|$, any collection of $r$ non-constant polynomials $p_1(x_1,\ldots,x_m),\ldots,p_r(x_1,\\
\ldots,x_m)$  of $m$ complex variables, with coefficients over $\CC$ and finite zero set 
\[
\mathbf{Z}(p_1,\ldots,p_r)=\{\mathbf{x}\in \mathbb{C}^m|p_j(\mathbf{x})=0,1\leq j\leq r\}\subset \CC^m,
\]
and any $m$-tuple $\mathbf{X}=(X_1,\ldots,X_m)\in {M_n(\mathbb{C})}^m$, of pairwise commuting normal matrix contractions, such that $\|p_j(X_1,\ldots,X_m)\|=0$ for each $1\leq j\leq r$. We prove the existence of paths between arbitrary $m$-tuples, that lie in the intersection of the set of $m$-tuples of pairwise commuting normal matrix contractions in ${M_n(\mathbb{C})}^m$, and the $\delta$-ball 
\[
B_\eth(\mathbf{X},\delta)=\{\mathbf{Y}\in M_n(\CC)^m|\eth(\mathbf{X},\mathbf{Y})<\delta\}
\]
centered at $\mathbf{X}$ for some $\delta>0$, with respect to $\eth$. 

Two of the key features of these matrix paths is that $\delta$ can be chosen independent of $n$, and that they are contained in the intersection of $B_\eth(\mathbf{X},\varepsilon)$, and the set $\mathbb{ZD}_n^m(p_1,\ldots,p_r)$ of $m$-tuples of pairwise commuting normal matrix contractions such that, $\|p_j(Y_1,\ldots,Y_m)\|=0$ for each $(Y_1,\ldots,Y_m)\in \mathbb{ZD}_n^m(p_1,\ldots,p_r)$ and each $1\leq j\leq r$. 

Let $\epsilon>0$ and $m\in \mathbb{Z}^+$ be given. The reason why independence on matrix size $n$ (uniformity) is important in this study, is that in many applications one needs to perform computations such as matrix inversion or matrix decompostion/factorization, with some approximations $\{\mathbf{X}_k\}_{k\geq 1}$ of sequences of matrix $m$-tuples $\{\mathbf{Y}_k\}_{k\geq 1}$ such that, $\mathbf{X}_k,\mathbf{Y}_k\in M_{n_k}(\mathbb{C})^m$ and $\|\mathbf{X}_k-\mathbf{Y}_k\|\leq \epsilon$ for $k\geq 1$, and $\{n_k\}_{k\geq 1}\subseteq\mathbb{Z}^+$ is an increasing sequence of positive integers. In these circumstances, one needs that some properties of $\{\mathbf{X}_k\}_{k\geq 1}$, to be (uniform) independent of the matrix size $n_k$. 

A common technique, implemented in order to achive uniformity in the approximation of the sequences mentioned before, consists in {\bf preconditioning} each $m$-tuple $\mathbf{X}_k\in M_{n_k}(\mathbb{C})^m$ of the original sequence $\{\mathbf{X}_k\}_{k\geq 1}$ to obtain a sequence $\{\hat{\mathbf{X}}_k\}_{k\geq 1}$, with some desirable {\bf artificial} properties. In these document we focus on spectral artificial properties, more specifically on {\bf eigenvalue clustering} in the sense of \cite{Applied_Spectrum_Clustering_1,Applied_Spectrum_Clustering_3,Applied_Spectrum_Clustering_4}. 

It is worth mentioning that 
some times, eigenvalue clustering appears naturally in numerical models, before any preconditioning has been performed, as an example, one can consider lossless Drude dispersive
metallic photonic crystals in the sense of \cite{Applied_Spectrum_Clustering_3}.

The main sources of motivation for the reserach reported in this document came from structure preserving perturbation theory 
in the sense of \cite{Perturbation_Precond_Matrix_Coeff}, specially by the effect that {\bf preconditioning} in the sense of \cite{Matrix_Algebras_Preconditioners,Approximate_Eigenvectors, Projective_Preconditioners,Preconditioning_Matrix_Algebras, Applied_Spectrum_Clustering_4,Wathen_Preconditioning}, has on the numerical solution of linear systems of equations and eigenvalue/diagonalization problems, which are two of the main problems in numerical linear algebra. Another sources of motivation were, the perturbation theory of matrix polynomials in the sense of \cite{weakly_normal_polynomials}, and the simultaneous block-diagonalization of matrices in the sense of \cite{sim_block_diag}.

In this document we combine the duality between matrix paths and numerical linear algebra algorithms studied in \cite{Chu_num_lin}, with the geometric approach to matrix perturbation theory presented in \cite{Geometric_matrix_preturbation_theory}, in order to derive a topological approach to the solution of normal matrix approximation problems, by interpreting structure preserving (and almost structure preserving) numerical approximation/refinement problems, as {\bf algebraically constrained} topological conectivity problems in the metric {\bf matrix} space $(M_n(\CC)^m,\eth)$.

Matrix sets like $\mathbb{ZD}_n^m(p_1,\ldots,p_r)$, are called algebraic matrix sets in this document. The path connectivity  properties of some important families of algebraic matrix sets will be studied in \S\ref{main_results}. In order to extend the applicability of these connectivity results, some connections with the approximation theory for matrix functions of several normal ({\bf preconditioned}) matrix variables are studied in \S\ref{extended_main_results}. 

One of the reasons for exending the results in \S\ref{main_results} to \S\ref{extended_main_results}, comes from the study of the notion of {\bf approximate solvability} and stability, of iterative (Krylov space based) and direct methods implemented to solve linear algebra problems numerically, specially when the computation is performed with finite precision in the sense of \cite{Matrix_Functions_Finite_Precision}. 

\section{Preliminaries and Notation}

\label{notation}

Given $r$ polynomials $p_1(x_1,\ldots,x_m),\ldots,p_r(x_1,\ldots,x_m)$ of $m$ complex variables ,with coefficients over $\CC$, denote by $\mathbf{Z}(p_1,\ldots,p_r)$ the subset of $\CC^m$ determined by the expression.
\begin{equation}
\mathbf{Z}(p_1,\ldots,p_r)=\{(x_1,\ldots,x_m)\in \CC^m\: |\: p_j(x_1,\ldots,x_m)=0,1\leq j\leq r\}
\label{zero_set_definition}
\end{equation}

We will write $M_{m,n}$ to denote the set $M_{m,n}(\mathbb{C})$ of $m\times n$ complex matrices, if $m=n$ we will write $M_n$, we write $M_n^m$ to denote the set of $m$-tuples of $n\times n$ complex matrices. The symbols $\mathbf{1}_n$ and $\mathbf{0}_{m,n}$ will be used to denote the identity matrix and the zero matrix in $M_n$ and $M_{m,n}$ respectively, if $m=n$ we will write $\mathbf{0}_{n}$. Given a matrix $A\in M_n$, we write $A^\ast$ to denote the conjugate transpose $\bar{A}^\top$ of $A$. 

A matrix $X\in M_n$ is said to be normal if $XX^\ast=X^\ast X$, a matrix $H\in M_n$ is said to be hermitian if $H^\ast=H$, a matrix $K\in M_n$ is said to be skew hermitian if $K^\ast=-K$, and a 
matrix $U\in M_n$ such that $U^\ast U=UU^\ast=\mathbf{1}_n$ is called unitary. We will write $\mathbf{i}$ to denote the number $\sqrt{-1}$. Given any set $S$, we will write $|S|$ to denote the number of elements of $S$, counted without multiplicity.

Let $(X,d)$ be a metric space. We say that $\tilde{X}_\delta\subset X$ is a $\delta$-dense subset of $X$ if for all $x\in X$ there exists $\tilde{x}\in \tilde{X}_\delta$ such that $d(x,\tilde{x})\leq\delta$.

Given two locally compact Hausdorff spaces $X,Y$ we write $C(X,Y)$ and $C^1(X,Y)$ to denote the sets of continuous and $C^1$ (differentiable) functions between $X$ and $Y$, respectively.

From here on $\|\cdot\|$ denotes the operator norm defined for any $A\in M_n$ by $\|A\|:=\sup_{\|x\|_2=1}\|Ax\|_2$, where $\|\cdot\|_2$ denotes the Euclidean norm in $\CC^n$. Let us denote by $\eth$ the metric in $M_n^m$ defined by $\eth:M_n^m\times M_n^m\to \RR^+_0,(\mathbf{S},\mathbf{T})\mapsto \max_j \|S_j-T_j\|$. We will write $\disk[m]$ and $\TT[m]$ to denote the $m$-dimensional closed unit disk and the $m$-dimensional torus respectively. Given any $x_0\in \CC^m$, we will write $B(x_0,r)$ to denote the $r$-ball 
$\{x\in \CC^m\: | \: \|x-x_0\|_2< r\}$ in $\CC^m$.

In this document by a matrix contraction we mean a matrix $X$ in $M_n$ such that $\|X\|\leq 1$.  A matrix $P\in M_n$ such that $P^\ast=P=P^2$ is called a projector or projection. Given two projectors $P$ and $Q$, if $PQ=QP=\mathbf{0}_n$ we say that $P$ and $Q$ are orthogonal. By an orthogonal partition of unity (OPU) in $M_n$, we mean a finite set of pairwise orthogonal projectors $\{P_j\}$ in $M_n\backslash \{\mathbf{0}_n\}$ such that 
$\sum_j P_j=\mathbf{1}_n$. We will omit the explicit reference to $M_n$ when it is clear from the context.

Given any two matrices $X,Y\in M_n$ we will write $[X,Y]$ and $\mathrm{Ad}[X](Y)$ to denote the operations 
$[X,Y]:=XY-YX$ and $\mathrm{Ad}[X](Y):=XYX^*$.

A {\bf $\ast$-homomorphism} $\varphi:M_n\to M_n$ in $M_n$ is a linear and multiplicative map wich satisfies 
$\varphi(X^\ast)=\varphi(X)^\ast$ for all $X$ in $M_n$. Given $U\in \U{n}$, it can be easily verified that the map $\psi:M_n\to M_n$ defined by $\psi:=\mathrm{Ad}[U]$ is a $\ast$-homomorphism, any $\ast$-homomorphism of this form will be called an inner $\ast$-homomorphism.

Given any map $\Psi:M_n\to M_n$ in $M_n$, we write $\hat{\Psi}$ to denote the extended map $\hat{\Psi}:M_n^m\to M_n^m$ in $M_n^m$ determined by the assignment 
$\hat{\Psi}:(X_1,\ldots,X_m)\mapsto(\Psi(X_1),\ldots,\Psi(X_m))$ for all $(X_1,\ldots,X_m)$ in $M_n^m$.

Let $\mathbb{GL}_n$ denote the set of invertible elements in $M_n$. Given a matrix $A\in M_n$, we write $\sigma(X)$ to denote the set $\{\lambda\in \CC \: | \: A-\lambda\mathbf{1}_n\notin \mathbb{GL}_n\}$ of eigenvalues of $A$, the set $\sigma(A)$ is called the spectrum of $A$.

\begin{definition}[$\circledast$ operation]
 Given two matrix paths $\alpha,\beta\in C([0,1],M_n^m)$ we write ${\alpha \circledast \beta}$ to denote the concatenation of $\alpha$ and $\beta$, which is 
 the matrix path defined in terms of $\alpha$ and $\beta$ by the expression,
 \[
  {\alpha\circledast \beta}(s):=
  \left\{
  \begin{array}{l}
   \alpha(2s),\:\: 0\leq s\leq \frac{1}{2},\\
   \beta(2s-1),\:\: \frac{1}{2}\leq s\leq 1.
  \end{array}
  \right.  
 \]
\end{definition}

Given a matrix $A\in M_n$, we will write $\mathscr{D}(A)$ or $\mathrm{diag}[a_{11},a_{22},\ldots,a_{nn}]$ to denote the diagonal matrix defined by the following operation.
\begin{eqnarray}
\mathscr{D}(A)&:=&\mathrm{diag}[a_{11},a_{22},\ldots,a_{nn}]\\
&=&
\left(
\begin{array}{cccc}
a_{11} & 0 & \cdots & 0\\
0 & a_{22} & \cdots & \vdots\\
\vdots & \ddots & \ddots & 0\\
0 & \cdots & 0 & a_{nn}
\end{array}
\right)
\end{eqnarray}

It can be seen that $\mathscr{D}(\mathscr{D}(A))=\mathscr{D}(A)$ for any $A\in M_n$, the map $\mathscr{D}$ will be called the full pinching. 

\begin{remark}
By pinching inequalities (in the sense of \cite{Bhatia_matrix_inequalities}) we will have that $\|\mathscr{D}(A)-\mathscr{D}(B)\|=\|\mathscr{D}(A-B)\|\leq \|A-B\|$ for any two matrices $A,B\in M_n$.
\end{remark}

It is often convenient to have $N$-tuples (or $2N$-tuples) of matrices with real spectra. For this purpose we use the following construction. If $\mathbf{X}=(\NS{X}{N})$ is a $N$-tuple of $n$ by $n$ matrices then we can always decompose $X_j$ in the form 
$X_j=X_{1j}+\mathbf{i}X_{2j}$ where the $X_{kj}$ all have real spectra, and are determined by the equations 
\begin{equation}
\left\{
\begin{array}{l}
X_{1j}=(X_j+X_j^\ast)/2,\\
X_{2j}=(X_j-X_j^\ast)/{(2\mathbf{i})},
\end{array}
\right.
\label{hermitian_partition_relations}
\end{equation}
for each $1\leq j\leq N$.

We write 
$\hat{\pi}(\mathbf{X}):=(X_{11},\ldots,X_{1N},X_{21},\ldots,X_{2N})$ and call $\hat{\pi}(\mathbf{X})$ a {\bf partition} of $\mathbf{X}$. If the $X_{kj}$ all commute we say that 
$\hat{\pi}(\mathbf{X})$ is a commuting partition, and if the $X_{kj}$ are simultaneously triangularizable $\hat{\pi}(\mathbf{X})$ is a triangularizable partition. If 
the $X_{kj}$ are all semisimple (diagonalizable) then $\hat{\pi}(\mathbf{X})$ is called a semisimple partition.

Given any matrix $X$ in $M_n$, we will write $\mathrm{Re}(X)$ and $\mathrm{Im}(X)$ to denote the hermitian matrices defined by the equations $\mathrm{Re}(X)=(X+X^\ast)/2$ and $\mathrm{Im}(X)=(X-X^\ast)/(2\mathbf{i})$. 

\begin{remark}
It is important to recall that for any $X$ in $M_n$, $X^\ast X=X X^\ast$ if and only if $\mathrm{Re}(X)\mathrm{Im}(X)=\mathrm{Im}(X)\mathrm{Re}(X)$.
\end{remark}

Given a $2m$-tuple $\mathbf{X}=(X_{11},\ldots,X_{1m},X_{21},\ldots,X_{2m})$ in 
$M_n^{2m}$, the $m$-tuple obtained by the operation 
$\upsilon(\mathbf{X}):=(X_{11}+\mathbf{i}X_{21},\ldots,X_{1m}+\mathbf{i}X_{2m})\in M_n^m$ will be called 
{\bf juncture} of $\mathbf{X}$.

We say that $N$ normal matrices $\NS{X}{N}\in M_n$ are {\em simultaneously diagonalizable} if there is a unitary matrix 
$Q\in M_n$ such that 
$Q^* X_jQ$ is diagonal for each $j=1,\ldots,N$. In this case, for $1\leq k\leq n$, let 
$\Lambda^{(k)}(X_j):=(Q^*X_jQ)_{kk}$ the $(k,k)$ element of $Q^*X_jQ$, and set 
$\Lambda^{(k)}(\NS{X}{N}):=(\Lambda^{(k)}(X_1),\ldots,\Lambda^{(k)}(X_N))$ in $\CC^N$. The set
\[
 \Lambda(\NS{X}{N}):=\{\Lambda^{(k)}(\NS{X}{N})\}_{1\leq k\leq N}
\]
is called the {\em joint spectrum of $\NS{X}{N}$ with respect to $Q$}, or just the {joint spectrum of $\NS{X}{N}$} for short, we will omit the explicit reference to $Q$ when it is clear from the context. The unitary matrix $Q$ is called a {\em joint diagonalizer} of $\NS{X}{N}$ in this document. 

Given a set $S\subseteq M_n^m$ of $m$-tuples of pairwise commuting normal matrices, we will write $\Lambda(S)$ to denote the set $\{\Lambda(X) \: |\: X\in S\}$, the set $\Lambda(S)$ will be called the joint spectra of $S$. We will write $\Lambda(X_j)$ to denote the diagonal matrix representation of the $j$-component of $\Lambda(\NS{X}{N})$, in other words 
we will have that
\[
 \Lambda(X_j)=\diag{\Lambda^{(1)}(X_j),\ldots,\Lambda^{(n)}(X_j)}.
\]

Given a $m$-tuple $\mathbf{X}=(X_1,\ldots,X_m)\in M_n^m$ of commuting normal matrices, any orthogonal projection $P\in M_n$ such that 
$X_jP=PX_j=\Lambda^{(r)}(X_j)P$ for each $1\leq j\leq m$ and some $1\leq r\leq n$, will be called a joint spectral projector of $\mathbf{X}$.

\section{Path Connectivity of Algebraic Normal Matrix Sets}
\label{matrix_sets}

\label{algebraic_matrix_sets}

\subsection{Algebraic Hermitian Matrix Sets}

For any $n\in \mathbb{Z}^+$, we will write $\mathbb{I}^m_n$ to denote the subset of $M_n^m$ determined by the 
following expression.
\begin{equation}
\mathbb{I}^m_n=\left\{(X_1,\ldots,X_m)\in M_n^m\: \left|\:\: \begin{array}{l}
X_j X_k-X_k X_j=\mathbf{0}_n,\\
X_j-X_j^\ast=\mathbf{0}_n,\\
\|X_j\|\leq 1
\end{array} 1\leq j,k\leq m\right.\right\}
\label{matrix_zero_set_definition}
\end{equation} 
the set $\mathbb{I}_n^m(p_1,\ldots,p_r)$ will be called a {\em matrix $m$-cube} in this document.

Given any $n\in \mathbb{Z}^+$ and any $r$ polynomials $p_1(x_1,\ldots,x_m),\ldots,p_r(x_1,\ldots,x_m)$ of $m$ complex variables with coefficients over $\mathbb{R}$, we will write $\mathbb{ZI}_n^m(p_1,\ldots,p_r)$ to denote the subset of $\mathbb{I}^m_n$ determined by the 
expression.
\begin{equation}
\mathbb{ZI}_n^m(p_1,\ldots,p_r)=\{(X_1,\ldots,X_m)\in \mathbb{I}^m_n\: | p_j(X_1,\ldots,X_m)=\mathbf{0}_n, 1\leq j\leq r\}
\label{matrix_zero_set_definition}
\end{equation}
the algebraic hermitian matrix set $\mathbb{ZI}_n^m(p_1,\ldots,p_r)$ will be called an {\em algebraic matrix $m$-cube} in this document.

\subsection{Algebraic Normal Matrix Sets}

For any $n\in \mathbb{Z}^+$, we will write $\mathbb{D}^m_n$ to denote the subset of $M_n^m$ determined by the 
following expression.
\begin{equation}
\mathbb{D}^m_n=\left\{(X_1,\ldots,X_m)\in M_n^m\: \left|\:\: \begin{array}{l}
X_j X_k-X_k X_j=\mathbf{0}_n,\\
X_jX_j^\ast-X_j^\ast X_j=\mathbf{0}_n,\\
\|X_j\|\leq 1
\end{array} 1\leq j,k\leq m\right.\right\}
\label{matrix_zero_set_definition}
\end{equation} 
the set $\mathbb{ZI}_n^m(p_1,\ldots,p_r)$ will be called a {\em matrix $m$-disk} in this document.

Given any $n\in \mathbb{Z}^+$ and any $r$ polynomials $p_1(z_1,\ldots,z_m),\ldots,p_r(z_1,\ldots,z_m)$ of $m$ complex variables with coefficients over $\mathbb{C}$, we will write $\mathbb{ZD}_n^m(p_1,\ldots,p_r)$ to denote the subset of $\mathbb{D}^m_n$ determined by the 
expression.
\begin{equation}
\mathbb{ZD}_n^m(p_1,\ldots,p_r)=\{(X_1,\ldots,X_m)\in \mathbb{D}^m_n\: | p_j(X_1,\ldots,X_m)=\mathbf{0}_n, 1\leq j\leq r\}
\label{matrix_zero_set_definition}
\end{equation}
the algebraic normal matrix set $\mathbb{ZD}_n^m(p_1,\ldots,p_r)$ will be called an {\em algebraic matrix $m$-disk} in this document.

\subsection{Uniform path connectivity of algebraic normal contractions}
\label{main_results} 

\begin{lemma}
\label{flat_connectivity}
Given any $2$ matrix $m$-tuples $(\NS{X}{m}),(\NS{Y}{m})\in \mathbb{I}^m_n$ such that $X_jY_k=Y_kX_j$ for each $1\leq j,k\leq m$, there is a path $\gamma\in C^1([0,1],\mathbb{I}^m_n)$ that satisfies the conditions,  
\[
\left\{
\begin{array}{l}
\gamma(0)=(\NS{X}{m}),\\
\gamma(1)=(\NS{Y}{m})
\end{array}
\right.
\]
together with the constraints, 
\[
\eth(\gamma(t),(\NS{Y}{m})) \leq \eth((\NS{X}{m}),(\NS{Y}{m}))
\]
for each $0\leq t\leq 1$.
\end{lemma}
\begin{proof}
Given $(\NS{X}{m}),(\NS{Y}{m})\in \mathbb{I}^m_n$ such that $X_jY_k=Y_kX_j$ for each $1\leq j,k\leq m$, we will have that for each $1\leq j\leq m$, each matrix path of the form $\gamma_j(t)=X_j+t(Y_j-X_j)$ satisfies the interpolating conditions $\gamma_j(0)=X_j$ and $\gamma_j(1)=Y_j$, together with the constraints
\begin{equation}
\|\gamma_j(t)-Y_j\|\leq(1-t)\|X_j-Y_j\|\leq \|X_j-Y_j\|
\label{flat_path_inequality}
\end{equation} 
for each $0\leq t\leq 1$.

Let us set $\gamma(t)=(\gamma_j(t))$, $0\leq t\leq 1$, it can be easily verified that $\gamma(t)\in\mathbb{I}^m_n$ for each $0\leq t\leq 1$. By definition of $\eth$ and as a consequence of $\eqref{flat_path_inequality}$ we can derive the following estimate,
\begin{eqnarray}
\eth(\gamma(t),(\NS{Y}{m}))&=&\max_{j}\|\gamma_j(t)-Y_j\|\nonumber\\
&\leq& \max_j\|X_j-Y_j\|=\eth((\NS{X}{m}),(\NS{Y}{m}))
\label{flat_path_inequality_2}
\end{eqnarray}
for each $0\leq t\leq 1$. This completes the proof.
\end{proof}

\begin{lemma}
 \label{existence_of_a refinement}
 Given any two OPU 
 $\mathcal{P}:=\{P_1,\ldots,P_r\}$ and $\mathcal{Q}:=\{Q_1,\ldots,Q_s\}$ in $M_n$, there is an OPU $\mathcal{R}:=\{R_1,\ldots,R_t\}$ of $M_n$ such that 
 $\mathrm{span}\: \{\mathcal{P},\mathcal{Q}\}\subseteq \mathrm{span}\: {\mathcal{R}}$ and 
 $|\mathcal{R}|\leq |\mathcal{P}|\:|\mathcal{Q}|$.
\end{lemma}
\begin{proof}
 Since $\mathcal{P},\mathcal{Q}\subset D$, by setting $R_{j,k}:=P_jQ_k$ and $\mathcal{R}:=\{R_{j,k}\}\backslash \{\mathbf{0}\}$ it can be seen that 
 $\mathcal{P},\mathcal{Q}\subseteq \mathrm{span} \:\{R_{j,k}\}$. Let us set $\mathcal{R}:=\{R_{j,k}\}$, it can 
 be seen that $|\mathcal{R}|\leq |\mathcal{P}||\mathcal{Q}|$ and 
 $\mathrm{span}\: \{\mathcal{P},\mathcal{Q}\}\subseteq \mathrm{span}\: {\mathcal{R}}$. This completes the proof.
\end{proof}

\begin{definition}{Projective Refinement.}{\rm ~Given any collection of OPU $\mathcal{P}_1=\{P_{1,j_1}\}_{j_1=1}^{r_1},$ $\ldots,\mathcal{P}_s=\{P_{s,j_s}\}_{j_s=1}^{r_s}$ such that $P_{k,j_k} P_{l,j_l}=P_{l,j_l} P_{k,j_k}$ for any $1\leq k,l \leq s$, each $P_{k,j_k}\in \mathcal{P}_k$ and each $P_{l,j_l}\in \mathcal{P}_l$. The set $\mathcal{R}(\mathcal{P}_1,\ldots,\mathcal{P}_s)$ defined by the expression
\begin{equation}
\mathcal{R}(\mathcal{P}_1,\ldots,\mathcal{P}_s)=\{P_{1,j_1}P_{2,j_2}\cdots P_{s,j_s}|P_{k,j_k}\in \mathcal{P}_k\}\backslash \{\mathbf{0}\},
\label{refinement_definition}
\end{equation}
will be called a projective refinement of $\mathcal{P}_1,\ldots,\mathcal{P}_s$.}
\end{definition}

By iterating on Lemma \ref{existence_of_a refinement} we can obtain the following corollary.

\begin{corollary}\label{existence_of_refined_OPU}
For any collection of simultaneously commuting OPU $\mathcal{P}_1,\ldots,\mathcal{P}_s$, we will have that $\mathcal{R}(\mathcal{P}_1,\ldots,\mathcal{P}_s)$ is an OPU.
\end{corollary}

\begin{lemma}{{\rm Projective Polar Decomposition.}} Given an orthogonal matrix partition of unity $\mathcal{P}=\{P_j\}_{j=1}^r$ in $M_n$, and give any matrix $X\in M_n$, there is a polar decomposition $X_{j,j}=V_j R_j$ of the matrix $X_{j,j}=P_j X P_j$ that satisfies the conditions $V_jV_j^*=V_j^*V_j=P_j$, $R_j\geq \mathbf{0}_n$, 
$P_jV_j=V_jP_j=V_j$, $P_jR_j=R_jP_j=R_j$ and $P_kV_j=P_kR_j=R_jP_k=V_jP_k=\mathbf{0}_n$, for $1\leq k,j\leq r$ with $k\neq j$.
\end{lemma}
\begin{proof}
By changing basis if necessary, we can assume that the elements of $\mathcal{P}$ are diagonal matrices. We will have that 
for each $j=1,\ldots,r$ there is a unitary (permutation) matrix $S_j\in M_n$, such that $S_jP_jS_j^\ast=\hat{P}_j$ with $\hat{P}_j=\I_{m_j}\oplus\mathbf{0}_{n-m_j}$, for some $1\leq m_j\leq n$. If we set $\psi_j:=\mathrm{Ad}[S_j]$, then each $\psi_j$ is a $\ast$-homomorphism. We will have that for each $j$:
\begin{eqnarray}
\hat{P}_jS_j^\ast X S_j\hat{P}_j=\hat{X}_{j,j}\hat{P}_{j}=\hat{X}_{j,j}
\label{projective_decomposition_1}
\end{eqnarray}
with $\hat{X}_{j,j}=\tilde{X}_{j,j}\oplus \mathbf{0}_{n-m_j}$ for some $\tilde{X}_{j,j}\in M_{m_j}$. Let $\tilde{X}_{j,j}=\tilde{V}_{j,j}\tilde{R}_{j,j}$ be the polar decomposition of 
$\tilde{X}_{j,j}$. Then we have that the matrices $\hat{V}_{j,j}=\tilde{V}_{j,j}\oplus \mathbf{0}_{n-m_j}$ and $\hat{R}_{j,j}=\tilde{R}_{j,j}\oplus \mathbf{0}_{n-m_j}$ satisfy the equations:
\begin{eqnarray}
\left\{
\begin{array}{l}
\hat{X}_{j,j}=\hat{V}_{j,j}\hat{R}_{j,j}\\
\hat{V}_j\hat{V}_j^*=\hat{V}_j^*\hat{V}_j=\hat{P}_j\\
\hat{R}_j\geq \mathbf{0}_n\\
\hat{P}_j\hat{V}_j=\hat{V}_j\hat{P}_j=\hat{V}_j\\
\hat{P}_j\hat{R}_j=\hat{R}_j\hat{P}_j=\hat{R}_j\\
\hat{P}_k\hat{V}_j=\hat{P}_k\hat{R}_j=\hat{R}_j\hat{P}_k=\hat{V}_j\hat{P}_k=\mathbf{0}_n
\end{array}
1\leq k,j\leq r, k\neq j.
\right. 
\label{commutativity_relations}
\end{eqnarray}
We can use $\psi_j$ together with \eqref{commutativity_relations} to \eqref{projective_decomposition_1} in order to obtain the following decomposition:
\begin{eqnarray}
X_{j,j}&=&P_jXP_j\\
&=&\psi_j(\hat{P}_j) X \psi(\hat{P}_j)\\
&=&\psi_j(\hat{P}_jS_j^\ast X S_j\hat{P}_j)\\
&=&\psi_j(\hat{X}_{j,j}\hat{P}_{j})\\
&=&\psi_j(\hat{V}_{j,j}\hat{R}_{j,j}\hat{P}_{j})\\
&=&\psi_j(\hat{V}_{j,j}\hat{R}_{j,j})\\
&=&\psi_j(\hat{V}_{j,j})\psi_j(\hat{R}_{j,j})
\label{generalized_block_decomposition}
\end{eqnarray}
Let us set $V_j=\psi_j(\hat{V}_j)$ and $R_j=\psi_j(\hat{R}_j)$. Since each $\psi_j$ preserves commutativity and positivity, by \eqref{commutativity_relations} we have that 
$V_jV_j^*=V_j^*V_j=P_j$, $R_j\geq \mathbf{0}_n$, 
$P_jV_j=V_jP_j=V_j$, $P_jR_j=R_jP_j=R_j$ for each $1\leq j\leq r$. By the previous commutativity relations we have that $V_j=P_jV_jP_j$ and $R_j=P_jR_jP_j$ for each 
$j$, then  $P_kV_j=P_kR_j=R_jP_k=V_jP_k=\mathbf{0}_n$ for $1\leq k,j\leq r$ with $k\neq j$. This completes the proof.
\end{proof}

We can generalize the proof of \cite[P.VI.6.6]{Bhatia_mat_book} to obtain the following lemma.

\begin{lemma}
\label{existence_of_almost_unit}
 Given a unitary $W$ and a normal contraction $D$ in $M_n$ for $n\geq 2$, if $D=\sum_{j=1}^r \alpha_jP_j$ is diagonal for $2\leq r\in \ZZ$ and $\NS{\alpha}{r}\in \disk$, the 
 set $\{P_j\}$ consists of pairwise orthogonal diagonal projections in $M_n$ such that $\sum_{j}P_j=\I_n$, and $\alpha_j\neq \alpha_k$ whenever $k\neq j$, then there 
 is a unitary matrix $Z\in M_n$ and a constant $C$ depending on $r$ and $\sigma(D)$ such that $[Z,D]=0$ and $\|\I_n-WZ\|\leq C\|WDW^*-D\|$.
\end{lemma}
\begin{proof}
 Since there are $r$ mutually orthogonal 
 projections $\mathbf{0}_n\leq P_1,\ldots,P_r\leq \I_n$ in $M_n$ such that $\sum_{j}P_j=\I_n$ and 
 $D:=\sum_{j}\alpha_jP_j$ with $\alpha_j\in \disk$. By setting 
 $W_{j,k}:=P_jWP_k$, we will have that $W$ has a decomposition $W=\sum_{j,k}W_{j,k}$ and it can be seen that 
\begin{eqnarray}
 \|WDW^*-D\|&=&\|WD-DW\|\\
            &=&\|\sum_{j,k}(\alpha_j P_jW_{j,k}-\alpha_k W_{j,k}P_k\|\\
            &=&\|\sum_{j,k}(\alpha_j-\alpha_k)W_{j,k}\|.
\end{eqnarray}
Hence, for $j\neq k$,
\begin{eqnarray}
 \|W_{j,k}\|&\leq& \frac{1}{|\alpha_j-\alpha_k|}\|WDW^*-D\|\\
            &\leq& \max_{j,k}\left\{\frac{1}{|\alpha_j-\alpha_k|}\right\}\|WDW^*-D\|.
\end{eqnarray}
Hence, by setting $s=\min_{j,k,j\neq k}|\alpha_j-\alpha_k|$ we will have that
\[
 \left\|W-\sum_{j}W_{j,j}\right\|\leq \frac{r(r-1)}{s}\|WDW^*-D\|.
\]
Let $X:=\sum_{j}W_{j,j}=\sum_{j}P_jWP_j$. Hence $\|X\|\leq \|W\|=1$. Let 
$W_{j,j}:=V_jR_j$ be the polar decomposition of $W_{j,j}$, with $V_jV_j^*=V_j^*V_j=P_j$, $R_j\geq \mathbf{0}_n$, 
$P_jV_j=V_jP_j=V_j$, $P_jR_j=R_jP_j=R_j$ and $P_kV_j=P_kR_j=R_jP_k=V_jP_k=\mathbf{0}_n$, if $k\neq j$. Then
\[
 \|W_{j,j}-V_j\|=\|R_j-P_j\|\leq \|R_j^2-P_j\|,
\]
since $R_j$ is a contraction. Let $V:=\sum_{j}V_j$. Then $V\in \U{n}$ and from the above inequality, we see that 
\[
 \|X-V\|\leq \|X^*X-\I_n\|=\|X^*X-W^*W\|.
\]
Hence,
\begin{eqnarray}
\|V-W\|&\leq&\|V-X\|+\|X-W\|\leq\|W-X\|+\|X^*X-W^*W\|\\
          &\leq&\|W-X\|+\|(X^*-W^*)X\|+\|W^*(X-W)\|\\
          &\leq&3\|W-X\|\leq \frac{3r(r-1)}{s}\|WDW^*-D\|.       
\end{eqnarray}
By setting $Z:=V^*$ and $C:=\frac{3r(r-1)}{s}$, it can be seen that $\|\I_n-WZ\|=\|V-W\|\leq C\|WDW^*-D\|$ and 
also that $[Z,D]=[V,D]=0$. This completes the proof.
\end{proof}

\begin{remark}\label{finite_order_remark}
 Given any normal contraction $D$ such that $p(D)=\mathbf{0}_n$ for some $p\in \CC[z]$ with $\deg(p)\leq r$, we will have that there are an integer $r'\leq r$, $r'$ complex numbers $\NS{\alpha}{{r'}}\in \disk$, and $r'$ pairwise orthogonal projections 
 $P_1,\ldots,P_{r'}$ such that, $p(\alpha_j)=0$, $\sum_jP_j=\I_n$ and $D=\sum_{j}\alpha_jP_j$. 
\end{remark}

\begin{lemma}
\label{existence_of_refined_almost_unit}
 Given a unitary $W$ and a collection normal contractions $D_1,\ldots,D_m$ in $M_n$ for $n\geq 2$, if each $D_k=\sum_{j=1}^{r_k} \alpha_{k,j}P_{k,j}$ is diagonal for $2\leq r_k\in \ZZ$ and $\{\alpha_{k,j}\}\subseteq \disk$, each 
 set $\{P_{k,j}\}$ consists of pairwise orthogonal diagonal projections in $M_n$ such that $\sum_{j}P_{k,j}=\I_n$, and $\alpha_{k,j}\neq \alpha_{k,j}$ whenever $l\neq j$, then there 
 is a unitary matrix $Z\in M_n$ and a constant $C$ depending on $m,r_1,\ldots,r_m$ and the spectra $\sigma(D_1),\ldots,\sigma(D_m)$ such that $[Z,D_k]=0$ and $\|\I_n-WZ\|\leq C\max_{1\leq k\leq m}\|WD_kW^*-D_k\|$.
\end{lemma}
\begin{proof}
We can apply Lemma \ref{existence_of_almost_unit} to $W$ and each $D_k$ to obtain for each $k$ a unitary matrix $Z_k$ that satisfies the conditions 
\begin{eqnarray}
\left\{
\begin{array}{l}
[Z_k,D_k]=\mathbf{0}_n\\
\|\I_n-WZ\|\leq C_k\|WDW^*-D\|
\end{array}
\right.
\label{approx_conditions_1}
\end{eqnarray}
where $C_k$ is a constant that depends on $r$ and $\sigma(D_k)$. By \eqref{approx_conditions_1} we will have that for each $P_{k,j}$
\begin{eqnarray}
\|WP_{k,j}W^\ast-P_{k,j}\|&=&\|WP_{k,j}-P_{k,j}W\|\nonumber\\
&\leq&\|WP_{k,j}-Z_kP_{k,j}\|+\|P_{k,j}Z_k-P_{k,j}W\|\nonumber\\
&\leq&2\|W-Z_k\|\nonumber\\
&\leq&2C_k\|WD_kW^\ast-D_k\|
\label{first_projective_inequality}
\end{eqnarray}
Let us consider a fixed but arbitrary element $P$ in the projective refinement $\mathcal{R}(\{P_{1,j_1}\},\\
\ldots,\{P_{m,j_m}\})$. We will have that $P=P_{1,j'_1},\ldots,P_{m,j'_m}$ with $P_{k,j'_k}\in \{P_{k,j_k}\}$ for each $1\leq k\leq m$. This implies that 
\begin{eqnarray}
\|WPW^\ast-P\|&=&\|WP_{1,j'_1}\cdots P_{m,j'_m}W^\ast-P_{1,j'_1}\cdots P_{m,j'_m}\|\nonumber\\
&\leq&\sum_{k=1}^m\|WP_{k,j'_k}W^\ast-P_{k,j'_k}\|
\label{refined_projective_inequality}
\end{eqnarray}
Combining \eqref{first_projective_inequality} and \eqref{refined_projective_inequality} we obtain the following estimate.
\begin{eqnarray}
\|WPW^\ast-P\|&\leq&2m\max_{1\leq k\leq m}C_k\max_{1\leq k\leq m}\|WD_kW^\ast-D_k\|
\label{second_refined_projective_inequality}
\end{eqnarray}
Let us set $\nu=\max_{1\leq k\leq m}\|WD_kW^\ast-D_k\|$. If $\nu<1/(2m\max_{1\leq k\leq m}C_k)$, then by \cite[Lemma 2.5.1]{Amenable_algebras_Lin} we have that \eqref{second_refined_projective_inequality} implies that for each $P$ in the projective refinement $\mathcal{R}(\{P_{1,j_1}\},\ldots,\{P_{m,j_m}\})$ there is a unitary $W_P\in M_n$ such that.
\begin{equation}
\left\{
\begin{array}{l}
WPW^\ast=W_P^\ast P W_P\\
\|\mathbf{1}_n-W_P\|\leq \sqrt{2}\|WPW^\ast-P\|
\end{array}
\right.
\label{first_projective_commutation_relation}
\end{equation}

We have that the matrix $Z_P=W_PW\in M_n$ is a unitary that satisfies the following commutation relation.
\begin{equation}
Z_P P=P Z_P
\label{second_projective_commutation_relation}
\end{equation}
Let us list $\mathcal{R}(\{P_{1,j_1}\},\ldots,\{P_{m,j_m}\})$ in the form $\mathcal{R}(\{P_{1,j_1}\},\ldots,\{P_{m,j_m}\})=\{P_1,\ldots,P_{N}\}$. By \eqref{first_projective_commutation_relation} and \eqref{second_refined_projective_inequality} we have that for each $P_j\in \mathcal{R}(\{P_{1,j_1}\},\ldots,\{P_{m,j_m}\})$, there is a unitary $Z_j=W_{P_j}W\in M_n$ such that
\begin{equation}
\left\{
\begin{array}{l}
Z_jP_j=P_jZ_j\\
\|W-Z_j\|=\|\mathbf{1}_n-W_{P_j}\|\leq \sqrt{2}\|WP_jW^\ast-P_j\|
\end{array}
\right.
\label{third_projective_commutation_relation}
\end{equation}
As a consequence of \eqref{third_projective_commutation_relation} and Corollary \ref{existence_of_refined_OPU} it can 
be easily verified that $\hat{Z}=\sum_{j=1}^N Z_j P_j$ is a unitary. Moreover, by \eqref{second_refined_projective_inequality} we can obtain the following estimate.
\begin{eqnarray}
\|W-\hat{Z}\|&=&\left\|(W-\hat{Z})\left(\sum_{j=1}^N P_j\right)\right\|\nonumber \\
       &=&\left\|\sum_{j=1}^N (WP_j-Z_jP_j)\right\|\nonumber \\
       &\leq&\sum_{j=1}^N \|(\mathbf{1}_n-W_{P_j})WP_j\|\\
       &\leq&\sum_{j=1}^N \|\mathbf{1}_n-W_{P_j}\|\nonumber\\
       &\leq& 2\sqrt{2}mN \max_{1\leq k\leq m} C_k \nu\nonumber
\label{fourth_projective_inequality}
\end{eqnarray}
Let us set 
\begin{equation}
C=2\sqrt{2}mN \max_{1\leq k\leq m} C_k=6\sqrt{2}mN \frac{\max_{1\leq k\leq m}r_k(r_k-1)}{\min_{1\leq k\leq m} \min_{1\leq j,l\leq r_k, j\neq l} |\alpha_{k,j}-\alpha_{k,l}|}
\label{Universal_constant_definition}
\end{equation}
and 
\begin{equation}
Z=\hat{Z}^\ast.
\label{Joint_commuting_unitary_approximation}
\end{equation}
We will have that $ZD_k=D_kZ$, and 
\[
\|\I_n-ZW\|=\|\hat{Z}-W\|\leq C\max_{1\leq k\leq m}\|WD_kW^*-D_k\|.
\]
This completes the proof.
\end{proof}

The following result was proved in \cite{Vides_homotopies}.

\begin{lemma}[Existence of isospectral approximants]
\label{Joint_spectral_variation_inequality_2}
 Given $\varepsilon>0$ there is $\delta> 0$ such that, for any $2$ families of $m$ pairwise commuting normal 
 matrices $\NS{X}{m}$ and $\NS{Y}{m}$ 
 which satisfy the constraints $\|X_j-Y_j\|\leq \delta$ for each $1\leq j\leq N$, there is a constant $K_m$ and a unitary $W\in \U{n}$ such that the inner $\ast$-homomorphism $\Psi=\mathrm{Ad}[W]$ satisfies the conditions: $\sigma(\Psi(X_j))=\sigma(X_j)$, $[\Psi(X_j),Y_j]=0$ and $\max\{\|\Psi(X_j)-Y_j\|,\|\Psi(X_j)-X_j\|\}\leq K_m\delta$, for each $1\leq j\leq N$.
\end{lemma}

\begin{remark}
The constant $K_m$ in the statement of Lemma \ref{Joint_spectral_variation_inequality_2} depends only on $m$.
\end{remark}

\begin{lemma}
\label{Existence_of_local_path_of_homomorphisms}
 Given any $\varepsilon\geq 0$ and $m$ non-constant polynomials $p_1(x),\ldots,p_m(x)$ over $\CC$, there is $\delta\geq 0$ such that for 
 any integer $n\geq 1$ and any $2$ $m$-tuples $(X_{1},\ldots,X_{m})$ ,$(Y_{1},\ldots,Y_{m})$ in $\mathbb{I}^m_n$ which satisfy the relations 
  \begin{eqnarray*}
  \left\{
  \begin{array}{l}
   p_j(X_{j})=p_j(Y_j)=\mathbf{0}_n,\: \: \: 1\leq j\leq m\\
   \eth((X_1,\ldots,X_m),(Y_1,\ldots,Y_{m}))\leq \delta,
  \end{array}
 \right.
 \end{eqnarray*}
there is a path $\{\Psi_t\}_{t\in [0,1]}$ $\ast$-homomorphisms $\Psi_t:M_n\to M_n$ such that the extented maps 
$\hat{\Psi}_t:M_n^m\to M_n^m$ satisfy the following relations,
 \[
 \left\{
 \begin{array}{l}
 \hat{\Psi}_0(Y_1,\ldots,Y_m)=(X_1,\ldots,X_m),\\
 \hat{\Psi}_1(Y_1,\ldots,Y_m)=(Y_1,\ldots,Y_m),\\
 \end{array}
 \right.
 \]
 together with the constraint
 \[
 \eth(\hat{\Psi}_t(Y_1,\ldots,Y_m),(Y_1,\ldots,Y_m))\leq \varepsilon,
 \]
 for each $0\leq t\leq 1$. 
\end{lemma}
\begin{proof}
 By changing basis if necessary, we can assume that $\NS{Y}{m}$ are diagonal matrices. By Lemma \ref{flat_connectivity}, the result is clear when $n=1$ or $|\sigma(X_j)|=|\sigma(Y_j)|=\deg(p_j)=1$, for each $1\leq j\leq m$. Without loss of generality let us assume that $\max_{1\leq j\leq m}\deg(p_j)\geq 2$, $\max_{1\leq j\leq m}|\sigma(X_j)|\geq 2$, $\max_{1\leq j\leq m}|\sigma(Y_j)|\geq 2$ and $n\geq 2$, let us set 
 $K:=\prod_{j=1}^m \deg(p_j)$ and $L:=\max_{1\leq j\leq m} \deg(p_j)$, and let us consider the sets $\mathbf{Z}(p_j)=\{z\in \disk| p_j(z)=0\}$, $1\leq j\leq m$. 
 
 By Lemma \ref{Joint_spectral_variation_inequality_2} we will have that there 
 are a constant $K_m$, a unitary $W\in M_n$ and an inner $\ast$-homomorphism $\Psi=\mathrm{Ad}[\hat{W}]:M_n\to M_n$ such that $[\Psi(X_j),Y_j]=0$ and $\|\Psi(X_j)-Y_j\|\leq K_m\delta$. Let $\varepsilon>0$ be given. It is enough to consider the case $\varepsilon< 4\sin(1/8)<1/2$. Since $\max\{\|X_j\|,\|Y_j\|\}\leq 1$ for each $1\leq j\leq m$, for the rest of the proof we will only consider the sets $\mathbf{Z}(p_j)\cap [-1,1]$, $1\leq j\leq r$. Let $h_p>0$ be a number chosen so that, 
\[
h_p\leq \frac{1}{3K_m}\min_{1\leq j\leq N}\{\min_{x,y\in \mathbf{Z}(p_j)\cap [-1,1]}\{|x-y|\:|\:x\neq y\}\}
\] 
since $\mathbf{Z}(p_j)\cap[-1,1]\subset [-1,1]$ for each $1\leq j\leq r$, we have that $h_p\leq 2$. We will have that there is 
 $\delta>0$ that can be chosen so that.
 \begin{eqnarray}
 \delta&\leq& \frac{2 h_p\arcsin(\varepsilon/4)}{3\pi\sqrt{2}m K_mKL(L-1)} \nonumber\\
 &<&  \frac{h_p\varepsilon}{2} 
  \leq \min\left\{\varepsilon,\frac{1}{3K_m}\min_{1\leq j\leq N}\left\{\min_{x,y\in \mathbf{Z}(p_j)\cap [-1,1]}\left\{|x-y|\:|\:x\neq y\right\}\right\}\right\}
 \label{delta_bound}
 \end{eqnarray}

Since 
 $\delta<\frac{1}{3K_m}\min_{1\leq j\leq N}\{\min_{x,y\in \mathbf{Z}(p_j)\cap [-1,1]}\{|x-y|\:|\:x\neq y\}\}$ and 
 $p_j(X_j)=p_j(\Psi(X_j))=p_j(Y_j)=\mathbf{0}_n$, we will have that $Y_j=\hat{W}X_j\hat{W}^*=\Psi(X_j)$, otherwise we get a contradiction. 
 
 By Remark \ref{finite_order_remark} for each $1\leq j\leq m$ there is an OPU $\{P_{j,k_j}\}$ such that $Y_j\in \mathrm{span}\: \{P_{j,k}\}$, and by Corollary \ref{existence_of_refined_OPU} we have that the projective refinement $\mathcal{P}:=\mathcal{R}(\{P_{1,k_1}\},\ldots,\{P_{m,k_m}\})=\{P_1,\ldots,P_{K'}\}$ is an OPU with $|\mathcal{P}|\leq {K'}\leq K$, such that $Y_j\in \mathrm{span}\: \mathcal{P}$ for each $1\leq j\leq m$. 
 
 By Lemma \ref{existence_of_refined_almost_unit}, \eqref{Universal_constant_definition} and \eqref{delta_bound}, there is a 
 unitary $Z$ that satisfies the constraint $\|Z-\hat{W}\|\leq \frac{4}{\pi}\arcsin(\varepsilon/4)$, together with the relations 
 $[Z,\Psi(X_j)]=[Z,Y_j]=0$, $1\leq j\leq N$. If we set $W:=\hat{W}^*Z$, we will have that
 \begin{equation}
 WY_jW^*=\hat{W}^*Y_j\hat{W}=\Psi^{-1}(Y_j)=X_j,
\label{alternative_morphism_representation}
\end{equation}  
 for each $1\leq j\leq m$. Moreover, as a consequence of the proof 
 of \cite[Theorem 5.2]{Bhatia_Spectral_Subspace}, we will have that there is a skew hermitian matrix $K\in M_n$ that satisfies the relations. 
 \begin{equation}
\left\{
\begin{array}{l}
e^{K}=W,\\
 \|K\|\leq \frac{\pi}{2}\|\mathbf{1}_n-W\|=\frac{\pi}{2}\|Z-\hat{W}\|\leq \frac{\pi}{2}\frac{4}{\pi}\arcsin(\varepsilon/2)=2\arcsin(\varepsilon/4)
\end{array}
\right.
 \label{Unitary_log_constraint}
 \end{equation}
Since for any $t\in [0,1]$, we will have that 
 \begin{equation}
 |1-e^{\mathbf{i}t}|=2\sin\left(\frac{t}{2}\right).
 \label{special_functional_calculus_identity}
 \end{equation}
 As a consequence of \eqref{Unitary_log_constraint} and \eqref{special_functional_calculus_identity} we will have that 
 if we set $W(t)=e^{tK}$ with $0\leq t\leq 1$, then $W(t)\in \U{n}$ for each $t\in [0,1]$, $W(0)=\mathbf{1}_n$, $W(1)=W$, and we can obtain the following estimate
 \begin{equation}
 \|\mathbf{1}_n-W(t)\|\leq 2\sin\left(\frac{t\|K\|}{2}\right)\leq 2\sin\left(\frac{\|K\|}{2}\right)\leq 2\sin\left(\arcsin(\varepsilon/4)\right)\leq \frac{\varepsilon}{2}
 \label{unitary_path_normed_constraint}
 \end{equation}
for each $t\in [0,1]$. 

Let us set $\Phi_t=\mathrm{Ad}[W(t)]$, we will have that $\Phi_0=\mathrm{id}_{M_n}$, and by \eqref{alternative_morphism_representation} we will have that
\begin{equation}
\Phi_1(Y_j)=X_j=\Psi^{-1}(Y_j),
\label{second_alternative_morphism_representation}
\end{equation}
for each $1\leq j\leq m$. Furthermore, as a consequence of \eqref{Unitary_log_constraint} and \eqref{unitary_path_normed_constraint} we will have that,
\begin{eqnarray}
 \|\Phi_t(Y_j)-Y_j\|&=&\|W(t)Y_j W(t)^\ast -Y_j\| \nonumber\\
                    &=&\|W(t)Y_j -Y_jW(t)\|\nonumber\\
                    &\leq&\|W(t)Y_j-Y_j\|+\|Y_j-Y_jW(t)\|\nonumber\\
                    &\leq&2\|\mathbf{1}_n-W(t)\|\leq\varepsilon
 \label{morphism_path_normed_constraint}
 \end{eqnarray}
 for each $0\leq t\leq 1$. If we set $\Psi_t=\Phi_{1-t}$, we will have that each path $\{\Psi_t(Y_j)\}_{t\in [0,1]}$ is differentiable with respect to $t$, and satisfies the relation 
\begin{equation}
\left\{
\begin{array}{l}
\Psi_0(Y_j)=X_j,\\
\Psi_1(Y_j)=Y_j,
\end{array}
\right.
\label{interpolating_condition}
\end{equation} 
together with the constraint 
\begin{equation}
\|\Psi_t(Y_j)-Y_j\|\leq \varepsilon,
\label{normed_constraint_component_path}
\end{equation} 
for each $1\leq j\leq m$ and each $t\in [0,1]$. By \eqref{interpolating_condition} and \eqref{normed_constraint_component_path}, we will have that
 \[
 \left\{
 \begin{array}{l}
 \hat{\Psi}_0(Y_1,\ldots,Y_m)=({\Psi}_0(Y_1),\ldots,{\Psi}_0(Y_m))=(X_1,\ldots,X_m),\\
 \hat{\Psi}_1(Y_1,\ldots,Y_m)=({\Psi}_1(Y_1),\ldots,{\Psi}_1(Y_m))=(Y_1,\ldots,Y_m),\\
 \end{array}
 \right.
 \]
and
 \[
 \eth(\hat{\Psi}_t(Y_1,\ldots,Y_m),(Y_1,\ldots,Y_m))=\max_{1\leq j\leq m}\|\Psi_t(Y_j)-Y_j\|\leq \varepsilon,
 \]
for each $0\leq t\leq 1$. This completes the proof.
\end{proof}

\begin{theorem}
\label{first_local_hermitian_algebraic_connectivity}
Given any $\varepsilon\geq 0$ and $r$ non-constant polynomials $p_1(x_1,\ldots,x_m),\ldots\\
,p_r(x_1,\ldots,x_m)$ of $m$ complex variables, with coefficients over $\CC$, and with finite zero set $\mathbf{Z}(p_1,\ldots,p_r)\subset [-1,1]^m$, there is $\delta\geq 0$ such that for 
 any integer $n\geq 1$ and any $2$ $m$-tuples $(X_{1},\ldots,X_{m})$ ,$(Y_{1},\ldots,Y_{m})$ in $\mathbb{I}^m_n$ which satisfy the relations 
  \begin{eqnarray*}
  \left\{
  \begin{array}{l}
   p_j(X_1,\ldots,X_m)=p_j(Y_1,\ldots,Y_m)=\mathbf{0}_n,\: \: \: 1\leq j\leq r\\
   \eth((X_1,\ldots,X_m),(Y_1,\ldots,Y_{m}))\leq \delta,
  \end{array}
 \right.
 \end{eqnarray*}
there is a path $\varphi=(\varphi_1,\ldots,\varphi_m)\in C([0,1],\mathbb{I}^m_n)$ that satisfies the following relations,
 \begin{equation}
  \left\{
 \begin{array}{l}
 \varphi(0)=(X_1,\ldots,X_m),\\
 \varphi(1)=(Y_1,\ldots,Y_m),\\
 \end{array}
 \right.
 \label{solvent_constraint_1}
 \end{equation}
 together with the constraints
 \begin{equation}
 \left\{
 \begin{array}{l}
 p_j(\varphi(t))=\mathbf{0}_n, \:\: 1\leq j\leq r,\\
 \eth(\varphi(t),(Y_1,\ldots,Y_m))\leq \varepsilon,
 \end{array}
 \right.
 \label{solvent_constraint_2}
 \end{equation}
 for each $0\leq t\leq 1$.
\end{theorem}
\begin{proof}
Given $r$ polynomials $p_1(x_1,\ldots,x_m),\ldots,p_r(x_1,\ldots,x_m)$ of $m$ complex variables, with coefficients over $\CC$, as in the statement of this theorem. Let us set $L=|\mathbf{Z}(p_1,\ldots,p_r)|<\infty$, then $\mathbf{Z}(p_1,\ldots,p_r)$ can be listed in the form 
\begin{equation}
\mathbf{Z}(p_1,\ldots,p_r)=\{(x_{j,1},\ldots,x_{j,m})\: | \: 1\leq j \leq L\}.
\label{zero_set_list_representation}
\end{equation}
 
For each $1\leq k\leq m$, let us set $Z_k=\{x_{j,k}\: | \: 1\leq j\leq L\}$, and let us write $\breve{Z}_k$ to denote the set consisting of all distinct numbers in $Z_k$ counted without multiplicity. We will have that 
$\breve{Z}_k\subseteq Z_k$, and that for each $x\in Z_k$ there is $y\in \breve{Z}_k$ such that $x=y$, for each $1\leq k\leq m$. Let us set 
\begin{equation}
\hat{p}_k(x_k)=\prod_{y\in \breve{Z}_k}(x_k-y), 
\label{artificial_polynomials}
\end{equation}
for each $1\leq k\leq m$. We will have that each $\hat{p}_k(x_k)$ is a polynomial over $\CC$ such that $1\leq \deg(\hat{p}_k)= |\breve{Z}_k|\leq |Z_k|$, and that $\hat{p}_k(x)=0$ for every $x\in Z_k$ and each $1\leq k\leq m$. 

Given any two $(X_{1},\ldots,X_{m})$ ,$(Y_{1},\ldots,Y_{m})$ in $\mathbb{I}^m_n$ as in the statement of this theorem, as a direct application of multivariate functional calculus for commuting matrices, we will have that 
for each $1\leq j \leq r$ and each $1\leq k\leq n$,
\begin{equation}
\left\{
\begin{array}{l}
p_j(\Lambda^{(k)}(X_1,\ldots,X_m))=p_j(\Lambda^{(k)}(X_1),\ldots,\Lambda^{(k)}(X_m))=0,\\
p_j(\Lambda^{(k)}(Y_1,\ldots,Y_m))=p_j(\Lambda^{(k)}(Y_1),\ldots,\Lambda^{(k)}(Y_m))=0.
\end{array}
\right.
\label{applied_functional_calculus}
\end{equation}
By \eqref{applied_functional_calculus}, we will have that $\Lambda(X_1,\ldots,X_m),\Lambda(Y_1,\ldots,Y_m)\subseteq \mathbf{Z}(p_1,\ldots,p_r)$, and this implies that,
\begin{equation}
\hat{p}_j(X_j)=\hat{p}_j(\Lambda(X_j))=\mathbf{0}_n=\hat{p}_j(\Lambda(Y_j))=\hat{p}_j(Y_j)
\label{artificial_annihilator_condition}
\end{equation}
for each $1\leq j\leq m$.

Given $\varepsilon>0$, by Lemma \ref{Existence_of_local_path_of_homomorphisms} applied to $\hat{p}_1,\ldots,\hat{p}_m$ and any $2$ $m$-tuples $(X_{1},\ldots,X_{m})$ ,$(Y_{1},\ldots,Y_{m})\in\mathbb{I}^m_n$ as in the statement of this theorem, we will have that there is a family of $\ast$-homomorphisms $\Psi_t:M_n\to M_n$ such that,
 \begin{equation}
 \left\{
 \begin{array}{l}
 \hat{\Psi}_0(Y_1,\ldots,Y_m)=(X_1,\ldots,X_m),\\
 \hat{\Psi}_1(Y_1,\ldots,Y_m)=(Y_1,\ldots,Y_m),\\
 \end{array}
 \right.
 \label{morphism_constraint_1}
 \end{equation}
and
 \begin{equation}
  \eth(\hat{\Psi}_t(Y_1,\ldots,Y_m),(Y_1,\ldots,Y_m))\leq \varepsilon,
 \label{morphism_constraint_2}
 \end{equation}
for each $0\leq t\leq 1$. Let us set $\varphi(t)=(\varphi_1(t),\ldots,\varphi_m(t))$, with $\varphi_j(t)=\Psi_t(Y_j)$ for $1\leq j\leq m$ and $0\leq t\leq 1$, by \eqref{morphism_constraint_1} and 
\eqref{morphism_constraint_2} we will have that.
 \begin{equation}
 \left\{
 \begin{array}{l}
 \varphi(0)=(X_1,\ldots,X_m),\\
 \varphi(1)=(Y_1,\ldots,Y_m),\\
 \eth(\varphi(t),(Y_1,\ldots,Y_m))\leq \varepsilon,
 \end{array}
 \right.
 \label{curved_path_condition_0}
 \end{equation}
 Furthermore, for each $1\leq i,j\leq m$, $1\leq k \leq r$ and each $0\leq t\leq 1$, we will have that.
 \begin{eqnarray}
 \varphi_j(t)^\ast=(\Psi_t(Y_j))^\ast=\Psi_t(Y_j^\ast)=\Psi_t(Y_j)=\varphi_j(t)
 \label{curved_path_condition_1}
 \end{eqnarray}
 
  \begin{eqnarray}
  \varphi_j(t)\varphi_i(t)-\varphi_i(t)\varphi_j(t)&=&\Psi_t(Y_j)\Psi_t(Y_i)-\Psi_t(Y_i)\Psi_t(Y_j)\nonumber \\
                                                   &=&\Psi_t(Y_jY_i-Y_iY_j)=\Psi_t(\mathbf{0}_n)=\mathbf{0}_n
 \label{curved_path_condition_2}
 \end{eqnarray}
 
   \begin{eqnarray}
  p_k(\varphi(t))&=&p_k(\Psi_t(Y_1),\ldots,\Psi_t(Y_m))\nonumber\\
  &=&\Psi_t(p_k(Y_1,\ldots,Y_m))=\Psi_t(\mathbf{0}_n)=\mathbf{0}_n
 \label{curved_path_condition_3}
 \end{eqnarray}
 By the definition of $\varphi$, and by \eqref{curved_path_condition_1} and \eqref{curved_path_condition_2}, we will have that $\varphi\in C^1([0,1],\mathbb{I}^m_n)$. And by \eqref{curved_path_condition_0} and \eqref{curved_path_condition_3}, we will have that the path $\varphi\in C^1([0,1],\mathbb{I}^m_n)$ satisfies the conditions \eqref{solvent_constraint_1} and \eqref{solvent_constraint_2}. This completes the proof.
\end{proof}

\begin{definition}
We say that a matrix set $\mathbb{S}^m_n\subseteq M_n^m$ is uniformly piecewise differentiably path connected with respect to the metric $\eth$, if given $\varepsilon>0$, there is $\delta>0$ such that, for any $\mathbf{X}\in \mathbb{S}^m_n$ and any $\mathbf{Y}\in \mathbb{S}^m_n\cap B_{\eth}(\mathbf{X},\delta)$, there is a piecewise $C^1$ path $\gamma\in C([0,1],\mathbb{S}^m_n)$ such that, $\gamma(0)=\mathbf{X}$, $\gamma(1)=\mathbf{Y}$ and $\gamma(t)\in B_\eth(\mathbf{X},\varepsilon)$, for each $0\leq t \leq 1$.
\end{definition}

\begin{corollary}
\label{first_connectivity_corollary}
Given $r$ non-constant polynomials $p_1(x_1,\ldots,x_m),\ldots,p_r(x_1,\ldots,x_m)$ of $m$ complex variables, with coefficients over $\CC$, and with finite zero set $\mathbf{Z}(p_1,\ldots,p_r)\\
\subset\RR^m$. The algebraic matrix $m$-cube $\mathbb{ZI}_n^m(p_1,\ldots,p_r)$ is uniformly piecewise differentiably path connected with respect to the metric $\eth$.
\end{corollary}
\begin{proof}
Given $\varepsilon>0$. Let us consider any $r$ non-constant polynomials $p_1(x_1,\ldots,x_m),\\
\ldots,p_r(x_1,\ldots,x_m)$ of $m$ complex variables, with coefficients over $\CC$, and with finite zero set $\mathbf{Z}(p_1,\ldots,p_r)\subset\RR^m$. By Theorem \ref{first_local_hermitian_algebraic_connectivity} we will have that 
there is $\delta>0$, such that for any $\mathbf{X}\in \mathbb{ZI}_n^m(p_1,\ldots,p_r)$ and any $\mathbf{Y}\in \mathbb{ZI}_n^m(p_1,\ldots,p_r)\cap B_{\eth}(\mathbf{X},\delta)$, there is a piecewise $C^1$ path $\gamma\in C([0,1],\mathbb{ZI}_n^m(p_1,\ldots,p_r))$ such that, $\gamma(0)=\mathbf{X}$, $\gamma(1)=\mathbf{Y}$ and $\gamma(t)\in B_\eth(\mathbf{X},\varepsilon)$, for each $0\leq t \leq 1$. This completes the proof.
\end{proof}

\begin{theorem}
\label{first_connectivity_application}
Given $r$ non-constant polynomials $p_1(z_1,\ldots,z_m),\ldots,p_r(z_1,\ldots,z_m)$ of $m$ complex variables, with coefficients over $\CC$, and with finite zero set $\mathbf{Z}(p_1,\ldots,p_r)\\
\subset\CC^m$. The algebraic matrix $m$-disk $\mathbb{ZD}_n^m(p_1,\ldots,p_r)$ is uniformly piecewise differentiably path connected with respect to the metric $\eth$.
\end{theorem}
\begin{proof}
Given $\varepsilon>0$. Let us consider any $r$ non-constant polynomials $p_1(x_1,\ldots,x_m),\\
\ldots,p_r(x_1,\ldots,x_m)$ of $m$ complex variables, with coefficients over $\CC$, and with finite zero set $\mathbf{Z}(p_1,\ldots,p_r)\subset\CC^m$. 

By elementary theory of complex valued functions of several complex variables, we will have that there are $2r$ polynomials $\mathrm{Re}(p_1),\ldots,\mathrm{Re}(p_m)$, $\mathrm{Im}(p_1),\ldots,\mathrm{Im}(p_m)$ in $2m$ real variables $\mathrm{Re}(z_1),\ldots,\mathrm{Re}(z_m)$, 
$\mathrm{Im}(z_1),\ldots,\mathrm{Im}(z_m)$, with coefficients over $\mathbb{R}$, such that $\mathbf{z}=(z_1,\ldots,z_m)\in \mathbf{Z}(p_1,\ldots,p_r)\subset \mathbb{C}^m$ if and only if $(\mathrm{Re}(\mathbf{z}),\mathrm{Im}(\mathbf{z}))\in \mathbf{Z}(\mathrm{Re}(p_1),\mathrm{Im}(p_1),\ldots,\mathrm{Re}(p_r),\mathrm{Im}(p_r))\subset \mathbb{R}^{2m}$, where $\mathrm{Re}(\mathbf{z})=(\mathrm{Re}(z_1),\ldots,\mathrm{Re}(z_m))$ and 
$\mathrm{Im}(\mathbf{z})=(\mathrm{Im}(z_1),\ldots,\mathrm{Im}(z_m))$.

Let us consider the maps $\imath:\mathbb{C}^m\to \mathbb{R}^{2m},\mathbf{z}\mapsto (\mathrm{Re}(\mathbf{z}),\mathrm{Im}(\mathbf{z}))$ and $\kappa:\mathbb{R}^{2m}\to \mathbb{C}^{m},\mathbf{x}\mapsto (x_1+\mathbf{i}x_{m+1},\ldots,x_m+\mathbf{i}x_{2m})$. By the arguments in the previous paragraph we have that there is a one-to-one correspondence between $\mathbf{Z}(p_1,\ldots,p_r)$ and $\mathbf{Z}(\mathrm{Re}(p_1),\mathrm{Im}(p_1),\ldots,\mathrm{Re}(p_r),\mathrm{Im}(p_r))$ induced by $\imath\circ \kappa$ and $\kappa\circ \imath$. This in turn implies that there is a one-to-one correspondence between $\mathbb{ZD}^m_n(p_1,\ldots,p_r)$ and $\mathbb{ZI}^m_n(\mathrm{Re}(p_1),\mathrm{Im}(p_1),\ldots,\mathrm{Re}(p_r),\mathrm{Im}(p_r))$, induced by the maps $\hat{\pi}\circ \upsilon$ and $\upsilon\circ \hat{\pi}$ defined in \S\ref{notation}.

By definition of $\hat{\pi}$ and $\upsilon$ and by the arguments in the previous paragraph, for any $\mathbf{X},\mathbf{Y}\in \mathbb{ZD}^m_n(p_1,\ldots,p_r)$, on one hand we will have that $\hat{\pi}(\mathbf{X}),\hat{\pi}(\mathbf{Y})\in \mathbb{ZI}^m_n(\mathrm{Re}(p_1),\mathrm{Im}(p_1),\ldots,\mathrm{Re}(p_r),\mathrm{Im}(p_r))$ and $\eth(\hat{\pi}(\mathbf{X}),\hat{\pi}(\mathbf{Y}))\leq \eth(\mathbf{X},\mathbf{Y})$. On the other hand, for any 
$\mathbf{S},\mathbf{T}\in \mathbb{ZI}^m_n(\mathrm{Re}(p_1),\mathrm{Im}(p_1),\ldots,\mathrm{Re}(p_r),\mathrm{Im}(p_r))$, $\upsilon(\mathbf{S}),\upsilon(\mathbf{T})\in \mathbb{ZD}^m_n(p_1,\ldots,p_r)$ and $\eth(\upsilon(\mathbf{X}),\upsilon(\mathbf{X}))\leq 2\eth(\mathbf{S},\mathbf{T})$.

By Theorem \ref{first_local_hermitian_algebraic_connectivity} and by the arguments in the previous paragraph, we will have that there is $\delta>0$, such that for any $\mathbf{X}\in \mathbb{ZD}_n^m(p_1,\ldots,p_r)$ and any $\mathbf{Y}\in \mathbb{ZD}_n^m(p_1,\ldots,p_r)\cap B_{\eth}(\mathbf{X},\delta)$, there is a piecewise $C^1$ path $\gamma_H\in C([0,1],\mathbb{ZI}_n^{2m}(\mathrm{Re}(p_1),\\
\mathrm{Im}(p_1),\ldots,\mathrm{Re}(p_r),\mathrm{Im}(p_r)))$ such that, $\gamma_H(0)=\hat{\pi}(\mathbf{X})$, $\gamma_H(1)=\hat{\pi}(\mathbf{Y})$ and $\gamma_H(t)\in B_\eth(\hat{\pi}(\mathbf{X}),\varepsilon/2)$, for each $0\leq t \leq 1$. This in turn implies that the path 
$\gamma=\upsilon\circ \gamma_H\in C([0,1],\mathbb{ZD}_n^m(p_1,\ldots,p_r))$, which is clearly piecewise $C^1$, satisfies the conditions $\gamma(0)=\mathbf{X}$, $\gamma(1)=\mathbf{Y}$ and $\gamma_H(t)\in B_\eth(\hat{\pi}(\mathbf{X}),\varepsilon)$, for each $0\leq t \leq 1$. This completes the proof.
\end{proof}

\subsection{Uniform path connectivity of nearly algebraic normal contractions}
\label{extended_main_results}

In order to extend the applicability of the results presented in \S\ref{main_results}, in this section we will solve some connectivity problems on what we will call {\em nearly algebraic matrix sets}. 

Given $\epsilon>0$, any $n\in \mathbb{Z}^+$ and any $r$ polynomials $p_1(x_1,\ldots,x_m),\ldots,p_r(x_1,\ldots,x_m)$ of $m$ complex variables with coefficients over $\mathbb{R}$, we will write $\mathbb{ZI}_{n,\epsilon}^m(p_1,\ldots,p_r)$ to denote the subset of $\mathbb{I}^m(n)$ determined by the 
expression.
\begin{equation}
\mathbb{ZI}_{n,\epsilon}^m(p_1,\ldots,p_r)=\{(X_1,\ldots,X_m)\in \mathbb{I}^m_n\: | \:\: \|p_j(X_1,\ldots,X_m)\|\leq \epsilon, 1\leq j\leq r\}
\label{matrix_zero_set_definition}
\end{equation}
the hermitian matrix nearly algebraic set $\mathbb{ZI}_{n,\epsilon}^m(p_1,\ldots,p_r)$ will be called an {\em $\epsilon$-nearly algebraic matrix $m$-cube} in this document.

Given $\epsilon>0$, any $n\in \mathbb{Z}^+$ and any $r$ polynomials $p_1(z_1,\ldots,z_m),\ldots,p_r(z_1,\ldots,z_m)$ of $m$ complex variables with coefficients over $\mathbb{C}$, we will write $\mathbb{ZD}_{n,\epsilon}^m(p_1,\ldots,p_r)$ to denote the subset of $\mathbb{D}^m_n$ determined by the 
expression.
\begin{equation}
\mathbb{ZD}_{n,\epsilon}^m(p_1,\ldots,p_r)=\{(X_1,\ldots,X_m)\in \mathbb{D}^m_n\: | \:\: \|p_j(X_1,\ldots,X_m)\|\leq \epsilon, 1\leq j\leq r\}
\label{matrix_zero_set_definition}
\end{equation}
the hermitian matrix algebraic set $\mathbb{ZD}_{n,\epsilon}^m(p_1,\ldots,p_r)$ will be called an {\em algebraic matrix $m$-disk} in this document.

Given $\delta>0$ and any collection of simultaneously commuting normal matrix contractions $X_1,\ldots,X_m\in M_n$, we will prove that there are, a collection of simultaneously commuting normal matrix contractions $\tilde{X}_1,\ldots,\tilde{X}_m\in M_n$ together with $m$-polynomials $p_1(x_1),\ldots,p_r(x_m)$ with coefficients over $\CC$, such that $\|X_j-\tilde{X}_j\|\leq \delta$, $X_j\hat{X}_k=\hat{X}_kX_j$, and $p_j(\tilde{X_j})=\mathbf{0}_n$, $1\leq j,k\leq m$. Moreover, we have that for each $1\leq j\leq m$, $\deg(p_j)$ does not depend on $n$. The collection $\tilde{X}_1,\ldots,\tilde{X}_m$ will be called {\em $\delta$-clustered pseusdospectral approximants} ({\bf CPA$_\delta$}) of $X_1,\ldots,X_m$. 

In terms of $m$-tuples in $\mathbb{D}^m_n$, we write $\mathbf{CPA}_\delta(X_1,\ldots,X_m)=\mathbf{CPA}_\delta(\tilde{X}_1,\ldots,\tilde{X}_m)$ to indicate that the components $\tilde{X}_1,\ldots,\tilde{X}_m$ of $m$-tuple $(\tilde{X}_1,\ldots,\tilde{X}_m)\in \mathbb{ZD}^m_n(p_1,\ldots\\
,p_j)\cap B_\eth((X_1,\ldots,X_m),\delta)$ are $\delta$-clustered pseusdospectral approximants of $X_1,\ldots,\\
X_m$. Using this notation it is enough to prove the following lemma, in order to solve the matrix approximation problem stated in the previous paragraph.

\begin{lemma}
\label{clustered_matrix_approximation_lemma}
Given any $(X_1,\ldots,X_m)\in \mathbb{D}^m_n$ .The problem $\mathbf{CPA}_\delta(X_1,\ldots,X_m)=(\tilde{X}_1,\ldots,\tilde{X}_m)$ is solvable for any $\delta>0$.
\end{lemma}
\begin{proof}
Let $\delta>0$ be given. Because of the one-to-one correspondence between $\mathbb{D}^m_n$ and $\mathbb{I}^{2m}_n$, induced by $\hat{\pi}\circ\upsilon$ and $\upsilon\circ \hat{\pi}$, together with the constraint $\eth(\mathbf{X},\mathbf{X}')\leq \eth(\upsilon(\hat{\pi}(\mathbf{X})),\upsilon(\hat{\pi}(\mathbf{X}')))$, that is satisfied for any $(\mathbf{X},\mathbf{X}')\in \mathbb{D}^m_n$. We have that it is enough to solve the problem in $\mathbf{CPA}_\delta(X_1,\ldots,X_{m})=(\tilde{X}_1,\ldots,\tilde{X}_{m})$ for any $(X_1,\ldots,X_{m})\in \mathbb{I}^{m}_n$.

Let $\mathbf{X}=(X_1,\ldots,X_{m})\in\mathbb{I}^m_n$ be given. Let us assume for simplicity that $1/\delta=N_\delta\in \mathbb{Z}^+$. Let us consider the grid $X_\delta:=\{x_k=-1+k\delta|0\leq k\leq N_\delta-1\}\subset [-1,1]$. It can be seen that the set $X_\delta$ is $\delta$-dense in $[-1,1]$. Morover, if $\chi_{[a,b)}$ denotes the characteristic function of the interval $[a,b)$, then it is clear that the simple function $\hat{p}(x)=\sum_{j=0}^{N_\delta-2}(-1+(j+1/2)\delta)\chi_{[x_j,x_{j+1})}$, is a $\delta$-approximation of the identity $\mathrm{id}$ function on $L^\infty([-1,1])$. 

By Borel (matrix) functional calculus we have that, if for each $1\leq j\leq m$ we set $\tilde{X}_j=\hat{p}(X_j)$, then $\|\tilde{X}_j-X_j\|= \|\hat{p}(X_j)-\mathrm{id}(X_j)\|\leq\|\hat{p}-\mathrm{id}\|_{L^\infty([-1,1])}\leq \delta$. By Borel functional calculus we also have that $[\tilde{X}_j,X_k]=[\hat{p}(X_j),X_k]=\mathbf{0}_n$ for each $1\leq j,k\leq m$.

Let us set $p(x)=\prod_{j=0}^{N_\delta-2}(x-(-1+(j+1/2)\delta))$, and let us write $p_j(x_j)$ to denote the minimal polynomial of $\tilde{X}_j$ for each $1\leq j\leq m$. By definition of each $\tilde{X}_j$, we have that $\deg(p_j)\leq \deg(p)< N_{\delta}$ does not depend on $n$ in general, in particular when $N_\delta\geq 1$. It can be seen that the $m$-tuple $(\tilde{X}_1,\ldots,\tilde{X}_m)\in \mathbb{ZI}_n^m(p_1,\ldots,p_m)$ solves the approximation problem $\mathbf{CPA}_\delta(X_1,\ldots,X_m)=(\tilde{X}_1,\ldots,\tilde{X}_m)$. This completes the proof.
\end{proof}

\begin{remark}
\label{existence_of_preconditioned_m_disks}
Given $\delta>0$, any $(X_1,\ldots,X_m)\in\mathbb{D}^m_n$, from the proof of Lemma \ref{clustered_matrix_approximation_lemma} we have that one can find two disjoint finite grids $X_\delta,X'_\delta\subset[-1,1]$ such that, by applying some retraction $\rho:[-1,1]\backslash X_\delta\to X'_{\delta}$ (if necessary), to each component $\tilde{X}_j$ of the solution of the problem $\mathbf{CPA}_\delta(X_1,\ldots,X_m)=(\tilde{X}_1,\ldots,\tilde{X}_m)$ determined by Lemma \ref{clustered_matrix_approximation_lemma}, it can be seen that we can obtain a {\em \bf preconditioned} $m$-tuple $(\hat{X}_1,\ldots,\hat{X}_m)\in \mathbb{ZD}_{n,\epsilon}(p_1,\ldots,p_r)$, where $\hat{X}_j=\rho(\tilde{X}_j)$, for each $1\leq j\leq m$, and $p_1,\ldots,p_r$ are determined by $\rho$, $X_\delta$ and $X'_\delta$.
\end{remark}

\begin{theorem}
\label{local_connectivity_m_tuples_almost_algebraic_commuting_contractions}
 Given any $\varepsilon\geq 0$ and $r$ non-constant polynomials $p_1(x_1,\ldots,x_m),\ldots\\
 ,p_r(x_1,\ldots,x_m)$ of $m$ complex variables, with coefficients over $\CC$, and with finite zero set $\mathbf{Z}(p_1,\ldots,p_r)\subset [-1,1]^m$, there are $\delta,\delta'> 0$ such that for 
 any integer $n\geq 1$ and any $2$ $m$-tuples $(X_{1},\ldots,X_{m})$ ,$(Y_{1},\ldots,Y_{m})$ in $\mathbb{I}^{m}(n)$ which satisfy the relations 
  \begin{eqnarray*}
  \left\{
  \begin{array}{l}
   \|p_j(X_{1},\ldots,X_m)\|\leq \delta',\\
   \|p_j(Y_1,\ldots,Y_m)\|\leq \delta',\\
   \eth((X_{1},\ldots,X_{m}),(Y_{1},\ldots,Y_{m}))\leq \delta,
  \end{array}
 \right.
 \end{eqnarray*}
for each $1\leq j \leq r$, there is a path $\varphi\in C^1([0,1],\mathbb{I}^m_n)$ that satisfies the following relations,
 \begin{equation}
  \left\{
 \begin{array}{l}
 \varphi(0)=(X_1,\ldots,X_m),\\
 \varphi(1)=(Y_1,\ldots,Y_m),\\
 \end{array}
 \right.
 \label{approximate_solvent_constraint_1}
 \end{equation}
 together with the constraints
 \begin{equation}
 \left\{
 \begin{array}{l}
 \|p_j(\varphi(t))\|< \delta', \:\: 1\leq j\leq r,\\
 \eth(\varphi(t),(Y_1,\ldots,Y_m))\leq \varepsilon,
 \end{array}
 \right.
 \label{approximate_solvent_constraint_2}
 \end{equation}
 for each $0\leq t\leq 1$.
\end{theorem}
\begin{proof}
Let $\varepsilon>0$ be given, and let us assume that $\varepsilon<4\sin(1/8)<1/2$. Given any polynomials $p_1,\ldots,p_r$, let us assume that $|\mathbf{Z}(p_1,\ldots,p_r)|\geq 2$, as the following argument can be easily modified when $|\mathbf{Z}(p_1,\ldots,p_r)|=1$. Since $\mathbf{Z}(p_1,\ldots,p_r)$ is finite, we will have that if we set 
\begin{equation}
\delta_1=\frac{1}{3}\min_{\mathbf{x},\mathbf{y}\in \mathbf{Z}(p_1,\ldots,p_r)}\{\|\mathbf{x}-\mathbf{y}\|_2\:|\: \mathbf{x}\neq\mathbf{y}\}>0,
\label{approximate_zero_set_delta_condition}
\end{equation}
then and for any 
$(x_1,\ldots,x_m),(x'_1,\ldots,x'_m)\in \mathbf{Z}(p_1,\ldots,p_r)$, we will have that $B((x_1,\ldots,x_m),\delta_1/2)\cap B((x'_1,\ldots,x'_m),\delta_1/2)=\emptyset$. 

By continuity of each $p_j(x_1,\ldots,x_m)$, we will have that for any $\delta'>0$ chosen so that $\delta'\leq  \min\{\delta_1,\varepsilon/2\}$, there is $\delta_2>0$ such that, for each $(x_1,\ldots,x_m)\in \mathbf{Z}(p_1,\ldots,p_r)$ and any $(y_1,\ldots,y_m)\in B((x_1,\ldots,x_m),\delta_2/2)$, we have that the inequality $|p_j(y_1,\ldots,y_m)|<\delta'$ is satisfied, for each $1\leq j\leq r$.

Let us set $\delta_3=\min\{\delta_1,\delta_2\}$. Given $\nu>0$, let us write $Z_\nu(p_1,\ldots,p_r)$ to denote the set determined by the following expression.
\begin{equation}
Z_\nu(p_1,\ldots,p_r)=\bigcup_{\mathbf{x}\in \mathbf{Z}(p_1,\ldots,p_r)} B(\mathbf{x},\nu)
\label{approximate_zero_set}
\end{equation}
Since $|\mathbf{Z}(p_1,\ldots,p_r)|\geq 2$, given $\tilde{\mathbf{x}}$ and $\tilde{\mathbf{y}}$ in the approximate zero set $Z_{\delta_3/2}(p_1,\ldots,p_r)$, such that $\tilde{\mathbf{x}}\in B(\mathbf{x},\delta_3/2)$ and $\tilde{\mathbf{y}}\in B(\mathbf{y},\delta_3/2)$ for some $\mathbf{x},\mathbf{y}\in \mathbf{Z}(p_1,\ldots,p_r)$, by \eqref{approximate_zero_set_delta_condition} and \eqref{approximate_zero_set} we will have that 
$\|\tilde{\mathbf{x}}-\tilde{\mathbf{y}}\|_2>\delta_2$. This implies that for any $\tilde{\mathbf{x}},\tilde{\mathbf{x}}'\in Z_{\delta_3/2}(p_1,\ldots,p_3)$ such that $\|\tilde{\mathbf{x}}-\tilde{\mathbf{x}}'\|<\delta_3$, there is $\mathbf{x}\in \mathbf{Z}(p_1,\ldots,p_r)$ such that $\tilde{\mathbf{x}},\tilde{\mathbf{x}}'\in B(\mathbf{x},\delta_3/2)$. This implies that for any $\mathbf{x},\mathbf{y}\in \mathbb{R}^m$ that satisfy the inequality $\|\mathbf{x}-\mathbf{x}\|_2<\delta_3$, together with the constraints $|p_j(\mathbf{x})|<\delta'$ and $|p_j(\mathbf{y})|<\delta'$ for each $1\leq j\leq r$, we will have that $\mathbf{x},\mathbf{y}\in Z_{\delta_3/2}(p_1,\ldots,p_r)$, otherwise we get a contradiction.

As a consequence of the arguments in the previous paragraph, we will have that for any two $\mathbf{X}=(X_1,\ldots,X_m)$ and $\mathbf{X}'=(X_1',\ldots,X'_m)$ in $\mathbb{I}^m_n$ that satisfy the inequality  $\eth(\mathbf{X},\mathbf{X}')<\delta_3$, together with the constraints $\|p_j(X_j,\ldots,X_m)\|<\delta'$ and 
$\|p_j(X'_j,\ldots,X'_m)\|<\delta'$ for each $1\leq j\leq r$, we will have that $\Lambda(X_1,\ldots,X_m)$, $\Lambda(X'_1,\ldots,X'_m)\in Z_{\delta_3/2}(p_1,\ldots,p_r)$.

Given any two $m$-tuples $\mathbf{H}=(H_1,\ldots,H_m)$, $\mathbf{H}'=(H'_1,\ldots,H'_m)$ in $\mathbb{I}^m_n$ such that $\Lambda(H_1,\ldots,H_m)$, $\Lambda(H'_1,\ldots,H'_m)\in Z_{\delta_3/2}(p_1,\ldots,p_r)$, let us consider a basis in which $\Lambda(H'_j)=\diag{\Lambda^{(1)}(H_j),\ldots,\Lambda^{(n)}(H_j)}=H'_j$ for each $1\leq j\leq m$. By Lemma \ref{Joint_spectral_variation_inequality_2}, we will have that there is a $\ast$-homomorphism $\Phi:M_n\to M_n$ that satisfies the conditions 
$\Lambda(\Phi(H_j))=\Phi(H_j)$ and $\Phi(H_j)\Lambda(H'_k)=\Lambda(H'_k)\Phi(H_j)$, together with the following constraints,
\begin{eqnarray}
\|\Lambda(\Phi(H_j))-\Lambda(H'_j)\|&\leq& \|(\Phi(H_j)-H_j\|+\|H_j-H'_j\|\nonumber\\
                                    &\leq& (K_m+1)\eth((H_1,\ldots,H_m),(H'_1,\ldots,H'_m))
\label{joint_spectral_variation_second_inequality}
\end{eqnarray}
for each $1\leq j,k\leq m$. By \eqref{joint_spectral_variation_second_inequality}, we will have that,
\begin{eqnarray}
\|\Lambda^{(k)}(\hat{\Phi}(\mathbf{H}))-\Lambda^{(k)}(\mathbf{H}')\|_2&\leq &\sqrt{m}\max_{1\leq j\leq m}|\Lambda^{(k)}(\Phi(H_j))-\Lambda^{(k)}(H'_j)|\nonumber\\
&\leq &\sqrt{m}\max_{1\leq j\leq m}\|\Lambda(\Phi(H_j))-\Lambda(H'_j)\|\nonumber\\
&\leq& \sqrt{m}(K_m+1)\eth(\mathbf{H},\mathbf{H}')
\label{approximate_joint_spectrum_constraint}
\end{eqnarray}
for each $1\leq k\leq n$. Let us set $\delta_4=\delta_3/(\sqrt{m}(K_m+1))$. By \eqref{approximate_joint_spectrum_constraint} and by the previous arguments, we will have that if in addition
\begin{equation}
\eth((H_1,\ldots,H_m),(H'_1,\ldots,H'_m))<\delta_4,
\label{additional_normed_constraint}
\end{equation}
there is $\hat{\mathbf{h}}_{k}=(\hat{h}_{k,1},\ldots,\hat{h}_{k,m})\in \mathbf{Z}(p_1,\ldots,p_r)$ such that,
\begin{equation}
\left\{
\begin{array}{l}
\|\hat{\mathbf{h}}_k-\Lambda^{(k)}(\hat{\Phi}(\mathbf{H}))\|_2<\delta_4/2,\\
\|\hat{\mathbf{h}}_k-\Lambda^{(k)}(\mathbf{H}')\|_2<\delta_4/2,
\end{array}
\right.
\label{zero_set_approximant_condition}
\end{equation} 
for each $1\leq k\leq n$. Let us set
\begin{equation}
\hat{H}'_j=\diag{\hat{h}_{1,j},\ldots,\hat{h}_{n,j}}
\label{matrix_zero_set_approximant_definition}
\end{equation}
for each $1\leq j\leq m$. We will have that the $m$-tuples $\hat{\mathbf{H}}'=(\hat{H}'_1,\ldots,\hat{H}'_m)$, $\hat{\mathbf{H}}=(\hat{H}_1,\ldots,\hat{H}_m)$ in $\mathbb{I}^m_n$, with $\hat{H}_j=\Phi^{-1}(\hat{H}'_j)$ for $1\leq j\leq m$, satisfy the relations,
\begin{equation}
\left\{
\begin{array}{l}
[\hat{H}_j,\hat{H}_k]=[\hat{H}'_j,\hat{H}'_k]=\mathbf{0}_n\\
\Lambda(\Phi(\hat{H}_j))=\Lambda(\hat{H}'_j),\\
p_k(\hat{H}_1,\ldots,\hat{H}_m)=p_k(\hat{H}'_1,\ldots,\hat{H}'_m)=\mathbf{0}_n,\\
\eth(\mathbf{H},\hat{\mathbf{H}})<\delta_3/2,\\
\eth(\mathbf{H}',\hat{\mathbf{H}}')<\delta_3/2,
\end{array}
\right.
\label{matrix_zero_set_approximants}
\end{equation}
for each $1\leq j\leq m$ and each $1\leq k\leq r$. It can be seen that $\hat{\mathbf{H}}'=\hat{\Phi}(\hat{\mathbf{H}})$, and by \eqref{matrix_zero_set_approximants}, we will have that.
\begin{eqnarray}
\eth(\hat{\mathbf{H}},\Phi(\hat{\mathbf{H}}))&\leq&\eth(\hat{\mathbf{H}},\mathbf{H})+\eth(\mathbf{H},\mathbf{H}')+\eth(\mathbf{H}',\hat{\mathbf{H}}')\nonumber\\
&<&\delta_3/2+\delta_4+\delta_3/2\leq 2\delta_3
\label{matrix_zero_set_approximants_conditions}
\end{eqnarray}

Let us consider $\mathbf{X}=(X_1,\ldots,X_m)$ and $\mathbf{Y}=(Y_1,\ldots,Y_m)$ in $\mathbb{I}^m_n$ that satisfy the inequality $\eth(\mathbf{X},\mathbf{Y})<\delta_4$ together with the constraints $\|p_j(X_1,\ldots,X_m)\|<\delta'$ and 
$\|p_j(Y_1,\ldots,Y_m)\|<\delta'$, for each $1\leq j\leq r$.  By the previous arguments together with \eqref{matrix_zero_set_approximants} and \eqref{matrix_zero_set_approximants_conditions}, we will have that there exist $\hat{\mathbf{X}}=(\hat{X}_1,\ldots,\hat{X}_m)$ and $\hat{\mathbf{Y}}=(\hat{Y}_1,\ldots,\hat{Y}_m)$ in $\mathbb{ZI}_n^m(p_1,\ldots,p_r)$ such that $\eth(\mathbf{X},\hat{\mathbf{X}})<\delta_3/2$, $\eth(\mathbf{Y},\hat{\mathbf{Y}})<\delta_3/2$ and $\eth(\hat{\mathbf{X}},\hat{\mathbf{Y}})<2\delta_3$. 

By Theorem \ref{first_local_hermitian_algebraic_connectivity} , by continuity of $p_1,\ldots,p_r$ and by \eqref{delta_bound}, we will have that there are, a number $\delta>0$ that can be chosen so that  $\delta<\min\{\delta',\varepsilon/2\}$, and a piecewise $C^1$ path $\phi\in C([0,1],\mathbb{ZI}^m_n(p_1,\ldots,p_r))$ such that if $\delta_3<\delta$, then  $\phi(0)=\hat{\mathbf{X}}$, $\phi(1)=\hat{\mathbf{Y}}$ and $\eth(\phi(t),\hat{\mathbf{Y}})<\varepsilon/2$ for each $t\in [0,1]$. 

As a consequence of Lemma \ref{flat_connectivity} we will have that there are piecewise $C^1$ paths $\phi_x,\phi_x\in C([0,1],\mathbb{I}^m_n)$ such that $\phi_x(0)=\mathbf{X}$, $\phi_x(1)=\hat{\mathbf{X}}$, $\phi_y(0)=\hat{\mathbf{Y}}$, $\phi_y(1)=\mathbf{Y}$, and in addition $\max\{\eth(\phi_x(t),\hat{\mathbf{X}}),\eth(\phi_y(t),\mathbf{Y})\}\leq \delta_3/2<\delta/2$. By definition of $\delta'$ and by continuity of each $p_j$ for $1\leq j\leq r$, we will have that $\max\{\|p_j(\phi_x(t))\|,\|p_j(\phi_y(t))\|\}\\
\leq \delta'$ for each $t\in [0,1]$.

Let us set $\varphi=(\phi_x\circledast \phi)\circledast \phi_y$. We will have that the path $\varphi\in C([0,1],\mathbb{I}^m_n)$ is piecewise $C^1$. By the arguments presented in the previous paragraphs we will have that $\|p_j(\varphi(t))\|<\delta'$, for each $1\leq j\leq r$ and each $t\in [0,1]$. Moreover, we will have that the path $\varphi$ satisfies the estimates,
\begin{eqnarray}
\eth(\varphi(t),\mathbf{Y})<\delta_3/2+\max\{\varepsilon/2,\delta_4\}+\delta_3/2\leq \varepsilon/2+\varepsilon/2=\varepsilon
\label{solvent_matrix_path_constraints}
\end{eqnarray}
for each $0\leq t\leq 1$. This completes the proof.
\end{proof}

\begin{corollary}
\label{second_connectivity_corollary}
Given  $r$ non-constant polynomials $p_1(x_1,\ldots,x_m),\ldots,p_r(x_1,\ldots,x_m)$ of $m$ complex variables, with coefficients over $\CC$, and with finite zero set $\mathbf{Z}(p_1,\ldots,p_r)\\
\subset\RR^m$. There is $\epsilon>0$ such that, for each $0<\epsilon'\leq \epsilon$,  the $\epsilon'$-nearly algebraic matrix $m$-cube $\mathbb{ZI}_{n,\epsilon'}^m(p_1,\ldots,p_r)$ is uniformly piecewise differentiably path connected with respect to the metric $\eth$.
\end{corollary}
\begin{proof}
Given $\varepsilon>0$. Let us consider any $r$ non-constant polynomials $p_1(x_1,\ldots,x_m),\\
\ldots,p_r(x_1,\ldots,x_m)$ of $m$ complex variables, with coefficients over $\CC$, and with finite zero set $\mathbf{Z}(p_1,\ldots,p_r)\subset\RR^m$. By Theorem \ref{local_connectivity_m_tuples_almost_algebraic_commuting_contractions} we will have that 
there are $\epsilon,\delta>0$, such that for any $\epsilon'\leq \epsilon$, any $\mathbf{X}\in \mathbb{ZI}_{n,\epsilon'}^m(p_1,\ldots,p_r)$ and any $\mathbf{Y}\in \mathbb{ZI}_{n,\epsilon'}^m(p_1,\ldots,p_r)\cap B_{\eth}(\mathbf{X},\delta)$, there is a piecewise $C^1$ path $\gamma\in C([0,1],\mathbb{ZI}_{n,\epsilon'}^m(p_1,\ldots,p_r))$ such that, $\gamma(0)=\mathbf{X}$, $\gamma(1)=\mathbf{Y}$ and $\gamma(t)\in B_\eth(\mathbf{X},\varepsilon)$, for each $0\leq t \leq 1$.
\end{proof}

\begin{theorem}
\label{second_connectivity_application}
Given  $r$ non-constant polynomials $p_1(x_1,\ldots,x_m),\ldots,p_r(x_1,\ldots,x_m)$ of $m$ complex variables, with coefficients over $\CC$, and with finite zero set $\mathbf{Z}(p_1,\ldots,p_r)\\
\subset\CC^m$. There is $\epsilon>0$ such that, for each $0<\epsilon'\leq \epsilon$,  the $\epsilon'$-nearly algebraic matrix $m$-disk $\mathbb{ZD}_{n,\epsilon'}^m(p_1,\ldots,p_r)$ is uniformly piecewise differentiably path connected with respect to the metric $\eth$.
\end{theorem}
\begin{proof}
Given $\varepsilon>0$. Let us consider any $r$ non-constant polynomials $p_1(x_1,\ldots,x_m),\\
\ldots,p_r(x_1,\ldots,x_m)$ of $m$ complex variables, with coefficients over $\CC$, and with finite zero set $\mathbf{Z}(p_1,\ldots,p_r)\subset\CC^m$. 

By elementary theory of complex valued functions of several complex variables, we will have that there are $2r$ polynomials $\mathrm{Re}(p_1),\ldots,\mathrm{Re}(p_m)$, $\mathrm{Im}(p_1),\ldots,\mathrm{Im}(p_m)$ in $2m$ real variables $\mathrm{Re}(z_1),\ldots,\mathrm{Re}(z_m)$, 
$\mathrm{Im}(z_1),\ldots,\mathrm{Im}(z_m)$, with coefficients over $\mathbb{R}$, such that $\mathbf{z}=(z_1,\ldots,z_m)\in \mathbf{Z}(p_1,\ldots,p_r)\subset \mathbb{C}^m$ if and only if $(\mathrm{Re}(\mathbf{z}),\mathrm{Im}(\mathbf{z}))\in \mathbf{Z}(\mathrm{Re}(p_1),\mathrm{Im}(p_1),\ldots,\mathrm{Re}(p_r),\mathrm{Im}(p_r))\subset \mathbb{R}^{2m}$, where $\mathrm{Re}(\mathbf{z})=(\mathrm{Re}(z_1),\ldots,\mathrm{Re}(z_m))$ and 
$\mathrm{Im}(\mathbf{z})=(\mathrm{Im}(z_1),\ldots,\mathrm{Im}(z_m))$.

Let us consider the maps $\imath:\mathbb{C}^m\to \mathbb{R}^{2m},\mathbf{z}\mapsto (\mathrm{Re}(\mathbf{z}),\mathrm{Im}(\mathbf{z}))$ and $\kappa:\mathbb{R}^{2m}\to \mathbb{C}^{m},\mathbf{x}\mapsto (x_1+\mathbf{i}x_{m+1},\ldots,x_m+\mathbf{i}x_{2m})$. By the arguments in the previous paragraph we have that there is a one-to-one correspondence between $\mathbf{Z}(p_1,\ldots,p_r)$ and $\mathbf{Z}(\mathrm{Re}(p_1),\mathrm{Im}(p_1),\ldots,\mathrm{Re}(p_r),\mathrm{Im}(p_r))$ induced by $\imath\circ \kappa$ and $\kappa\circ \imath$. This in turn implies that there is a one-to-one correspondence between $\mathbb{ZD}^m_n(p_1,\ldots,p_r)$ and $\mathbb{ZI}^m_n(\mathrm{Re}(p_1),\mathrm{Im}(p_1),\ldots,\mathrm{Re}(p_r),\mathrm{Im}(p_r))$, induced by the maps $\hat{\pi}\circ \upsilon$ and $\upsilon\circ \hat{\pi}$ defined in \S\ref{notation}.

By definition of $\hat{\pi}$ and $\upsilon$ and by the arguments in the previous paragraph, for any $\mathbf{X},\mathbf{Y}\in \mathbb{ZD}^m_n(p_1,\ldots,p_r)$, on one hand we will have that $\hat{\pi}(\mathbf{X}),\hat{\pi}(\mathbf{Y})\in \mathbb{ZI}^m_n(\mathrm{Re}(p_1),\mathrm{Im}(p_1),\ldots,\mathrm{Re}(p_r),\mathrm{Im}(p_r))$ and $\eth(\hat{\pi}(\mathbf{X}),\hat{\pi}(\mathbf{Y}))\leq \eth(\mathbf{X},\mathbf{Y})$. On the other hand, for any 
$\mathbf{S},\mathbf{T}\in \mathbb{ZI}^m_n(\mathrm{Re}(p_1),\mathrm{Im}(p_1),\ldots,\mathrm{Re}(p_r),\mathrm{Im}(p_r))$, $\upsilon(\mathbf{S}),\upsilon(\mathbf{T})\in \mathbb{ZD}^m_n(p_1,\ldots,p_r)$ and $\eth(\upsilon(\mathbf{X}),\upsilon(\mathbf{X}))\leq 2\eth(\mathbf{S},\mathbf{T})$.

Given $\nu>0$. By continuity of $p_1,\ldots,p_r$, on one hand, we have that for any $\mathbf{X}=(X_1,\ldots,X_m)\in \mathbb{ZD}^m_{n,\nu}(p_1,\ldots,p_r)$, $\max\{\|\mathrm{Re}(p_j(\hat{\pi}(\mathbf{X})))\|,\|\mathrm{Im}(p_j(\hat{\pi}(\mathbf{X})))\|\}\leq (1/2)(\|p_j(X_1,\ldots,X_m)\|
+\|(p_j(X_1,\ldots,X_m))^\ast\|)\leq \|p_j(X_1,\ldots,X_m))\|\leq\nu$, $1\leq j\leq r$. On the other hand, we have that for any $\mathbf{H}=(H_1,\ldots,H_{2m})\in \mathbb{ZI}^{2m}_{n,\nu/2}(\mathrm{Re}(p_1),\\
\mathrm{Im}(p_1),\ldots
,\mathrm{Re}(p_r),\mathrm{Im}(p_r))$, $\|p_j(\upsilon(\mathbf{H}))\|\leq \|\mathrm{Re}(p_j(\mathbf{H}))\|+\|\mathrm{Im}(p_j(\mathbf{H}))\|\}\leq (2\nu/2)=\nu$, $1\leq j\leq r$.

By Theorem \ref{local_connectivity_m_tuples_almost_algebraic_commuting_contractions} and by the arguments in the previous paragraphs, we will have that there are $\delta,\epsilon>0$, such that for any $\epsilon'\leq \epsilon$, any $\mathbf{X}\in \mathbb{ZD}_{n,\epsilon'}^m(p_1,\ldots,p_r)$ and any $\mathbf{Y}\in \mathbb{ZD}_{n,\epsilon'}^m(p_1,\ldots,p_r)\cap B_{\eth}(\mathbf{X},\delta)$, there is a piecewise $C^1$ path $\gamma_H\in C([0,1],\mathbb{ZI}_{n,\epsilon'/2}^{2m}(\mathrm{Re}(p_1),\mathrm{Im}(p_1),\ldots,\mathrm{Re}(p_r),\mathrm{Im}(p_r)))$ such that, $\gamma_H(0)=\hat{\pi}(\mathbf{X})$, 
\\
$\gamma_H(1)=\hat{\pi}(\mathbf{Y})$ and $\gamma_H(t)\in B_\eth(\hat{\pi}(\mathbf{X}),\varepsilon/2)$, for each $0\leq t \leq 1$. This in turn implies that the path 
$\gamma=\upsilon\circ \gamma_H\in C([0,1],\mathbb{ZD}_{n,\epsilon'}^m(p_1,\ldots,p_r))$, which is clearly piecewise $C^1$, satisfies the conditions $\gamma(0)=\mathbf{X}$, $\gamma(1)=\mathbf{Y}$ and $\gamma_H(t)\in B_\eth(\hat{\pi}(\mathbf{X}),\varepsilon)$, for each $0\leq t \leq 1$. This completes the proof.
\end{proof}

\section{Hints and Future Directions}
We will study the potential extension of our techniques to almost normal matrices, normal approximants, normal dilations and normal compressions. We will explore the computability of projective refinemets and approximate joint diagonalizers, together with their impact on the $\varepsilon-\delta$ relations, presented in the main connectivity results in \S\ref{main_results} and \S\ref{extended_main_results}.

\vspace{2pc}

\bigskip

{\bf Acknowledgment.} 
I would like to thank the Fields Institute for Research in Mathematical Sciences at the University of Toronto, for the support received during my research visit to the institute from July 31 to August 4, 2017. Much of the research reported in this document was carried out while I was visiting the Institute.

I am grateful with Terry Loring, Stanly Steinberg, Alexandru Buium, Moody Chu, Daniel Appel\"o, Alexandru Chirvasitu and Concepci\'on Ferrufino, for several interesting questions and comments that have been very helpful for the preparation of this document.

%%%%%%%%%%%%%%%%%%%%%%%%%%%%%%%%%%%%%%%%%%%%%%%%%%%%%%%%%%%%%

\end{document}